\documentclass[a4paper, 11pt,reqno]{amsart}

\usepackage{amsmath,amsthm,amsbsy,amssymb}
\usepackage{booktabs} 
\usepackage{verbatim} 


\usepackage{stmaryrd} 
\newcommand{\lOddB}{\big\llbracket} 
\newcommand{\rOddB}{\big\rrbracket} 

\usepackage{paralist} 
\usepackage{array} 

\makeatletter
\newcommand{\leqnos}{\tagsleft@true\let\veqno\@@leqno}
\newcommand{\reqnos}{\tagsleft@false\let\veqno\@@eqno}
\reqnos
\makeatother

\usepackage[backref=page]{hyperref}
\hypersetup{
    colorlinks = true,
linkcolor={red},
urlcolor={blue},
citecolor={green},    
urlcolor = {blue},
citebordercolor = {0.33 .58 0.33},
 linkbordercolor = {0.99 .28 0.23},
 breaklinks=true
}
\makeatletter
   \g@addto@macro{\UrlBreaks}{\do\/\do-\do\&\do.\do:}     
\makeatother

\usepackage{url}
\makeatletter
\def\url@leostyle{%
  \@ifundefined{selectfont}{\def\UrlFont{\sf}}{\def\UrlFont{\footnotesize\ttfamily}}}
\makeatother
\usepackage[user,titleref]{zref} 

\usepackage{arydshln} 
\usepackage{arydshln}

\usepackage{cancel}
\usepackage{xcolor}
\newcommand\Ccancel[2][black]{
    \let\OldcancelColor\CancelColor
    \renewcommand\CancelColor{\color{#1}}
    \cancel{#2}
    \renewcommand\CancelColor{\OldcancelColor}
}

\usepackage{mathrsfs}
\def\scrC{\mathscr C}

\newcommand{\cE}{\mathcal E}

\newcommand{\cH}{\mathcal H}

\newcommand{\cL}{\mathcal L}

\newcommand{\cP}{\mathcal P}

\newcommand{\fe}{\mathfrak e}
\newcommand{\fii}{\mathfrak i}

\newcommand{\fs}{\mathfrak s}
\newcommand{\fu}{\mathfrak u}

\newcommand{\fw}{\mathfrak w}

\usepackage{bm}
\newcommand*{\B}[1]{\bm{#1}}

\newcommand{\bw}{\B{w}}

\newcommand{\NN}{\mathbb{N}}
\newcommand{\ZZ}{\mathbb{Z}}

\newcommand{\FF}{\mathbb{F}}

\newcommand{\sdfrac}[2]{\mbox{\small$\displaystyle\frac{#1}{#2}$}}

\newcommand{\sdbinom}[2]{\mbox{\small$\displaystyle\binom{#1}{#2}$}}

\usepackage{xspace}

\DeclareMathOperator*{\ind}{ind}

\renewcommand{\pmod}[1]{\left( \mathrm{ mod\;}#1\right)}

\theoremstyle{plain}
\newtheorem{theorem}{Theorem}
\newtheorem{lemma}{Lemma}[section]

\newtheorem*{question*}{Question}
\newtheorem{proposition}{Proposition}[section]
\newtheorem{conjecture}{Conjecture}

\theoremstyle{definition}
\newtheorem{remark}{Remark}[section]
\newtheorem*{remark*}{Remark}

\newtheorem*{example*}{Example} 
\newtheorem*{problem*}{Problem} 
\newtheorem{problem}{Problem} 

\usepackage{subfigure}  
\usepackage{float}
\usepackage[pdftex]{graphicx}
\graphicspath{{IMAGES-ALL/}{IMAGES/}{IMAGES2/}{./}}
\usepackage[]{caption}

\newcommand{\mbinom}[2]{\mbinom{#1}{#2}}

\renewcommand{\mbinom}[3][0.8]{\scalebox{#1}{$\dbinom{#2}{#3}$}} 
\newcommand{\mfrac}[3][1.195]{\scalebox{#1}{$\frac{#2}{#3}$}}

\makeatletter
\@namedef{uubjclassname@2020}{\textup{2020} Mathematics Subject Classification}
\begin{document}


\title[Filtered rays over iterated absolute differences on layers of integers]
{Filtered rays over iterated absolute differences on layers of integers}

\author{Raghavendra N. Bhat}
\address[Raghavendra N. Bhat]{Department of Mathematics, University of Illinois, 1409 West Green 
Street, Urbana, IL 61801, USA}
\email{rnbhat2@illinois.edu}

\author[Cristian Cobeli]{Cristian Cobeli}
\address[Cristian Cobeli]{"Simion Stoilow" Institute of Mathematics of the Romanian Academy,~21 Calea Grivitei Street, P. O. Box 1-764, Bucharest 014700, Romania}
\email{cristian.cobeli@imar.ro}

\author{Alexandru Zaharescu}
\address[Alexandru Zaharescu]{Department of Mathematics, University of Illinois, 1409 West Green 
Street, Urbana, IL 61801, USA,
%
and 
"Simion Stoilow" Institute of Mathematics of the Romanian Academy,~21 
Calea Grivitei 
Street, P. O. Box 1-764, Bucharest 014700, Romania}
\email{zaharesc@illinois.edu}

\subjclass[2020]{Primary 11B37; Secondary 11B39, 11B50.}

\thanks{Key words and phrases:
Proth-Gilbreath Conjecture, formal power series, SP-numbers, gaps between primes.}

\begin{abstract}
The dynamical system generated by the iterated calculation of the high order gaps 
between neighboring terms of a sequence of natural numbers is remarkable and only 
incidentally characterized at the boundary by the notable Proth-Glibreath Conjecture 
for prime numbers.

We introduce a natural extension of the original triangular arrangement, 
obtaining a growing hexagonal covering of the plane. 
This is just the base level of what further becomes an endless discrete helicoidal surface.
Although the repeated calculation of higher-order gaps causes the numbers 
that generate the helicoidal surface to decrease, 
there is no guarantee, and most often it does not even happen, 
that the levels of the helicoid have any regularity, at least at the bottom levels.

However, we prove that there exists a large and nontrivial class of sequences 
with the property that their helicoids have all levels coinciding with their base levels. 
This class includes in particular many ultimately binary sequences with a special header.
For almost all of these sequences, we additionally show that although the patterns 
generated by them seem to fall somewhere between ordered and disordered, 
exhibiting fractal-like and random qualities at the same time, 
the distribution of zero and non-zero numbers at the base level has uniformity characteristics. 
Thus, we prove that a multitude of straight lines that traverse the patterns 
encounter zero and non-zero numbers in almost equal proportions.

\end{abstract}
\maketitle

\section{Introduction}
Let $\fu=\{a_k\}_{k\ge 0}$ be a sequence of non-negative integers.
We place the sequence $\fu$ on the top row of a triangle whose subsequent rows 
are recursively obtained as sequences of numbers given  by the absolute values 
of the differences between neighboring terms on the previous line.
The infinite equilateral triangle obtained in this way is defined by
\begin{equation}
\tag{P-G}\label{eqPGTriangle}
\setlength{\tabcolsep}{3pt}
   \begin{tabular}{ccccccccccccccccccccc}
 & $a_0$ && $a_1$ && $a_2$ && $a_3$ && $a_4$ && $a_5$ && $a_6$ && $\dots$ &\\[2mm]
& &  $d_0^{(1)}$   & &  $d_1^{(1)}$ &&  $d_2^{(1)}$ &&  $d_3^{(1)}$ &&  $d_4^{(1)}$ &&  $d_5^{(1)}$   && $\dots$ \\[2mm]
  &   && $d_0^{(2)}$  && $d_1^{(2)}$ && $d_2^{(2)}$ && $d_3^{(2)}$ && $d_4^{(2)}$ && $\dots$ &  &\\[2mm]  
   & &   &&  $d_0^{(3)}$ &&  $d_1^{(3)}$ &&  $d_2^{(3)}$ &&  $d_3^{(3)}$ &&  $\dots$   &  \\[2mm]
     &   &&  && $\dots$  && $\dots$  && $\dots$  && $\dots$  && &  &
  \end{tabular}
\end{equation}
where
\begin{equation}\label{eqDifferences}
   d_k^{(j+1)} := \big| d_{k+1}^{(j)} - d_{k}^{(j)}\big|
   \quad \text{ and } \quad d_k^{(0)} := a_k \quad \text{ for  $j, k\ge 0$.}\\[1mm]
\end{equation}
Let $\fw=\{b_j\}_{j\ge 0}$ be the sequence of numbers on the left edge of this triangle, that is,
$b_0=a_0$ and
$b_j = d_0^{(j)}$ for $j\ge 1$.
We also denote by $\fw_k=\big\{d_k^{(j)}\big\}_{j\ge 0}$, for $k\ge 0$, the column or the \textit{ray} that passes through the triangle parallel 
to the edge on the left, with the first component being $a_k$, so that $\fw=\fw_0$.

Taking successively the absolute values of higher-order differences, 
the resulting numbers on the lower rows become smaller and smaller. 
It is then natural to observe this comprehensive phenomenon on the left edge. 
However, the components of $\fw$ may not necessarily become all equal to $0$, 
but we should expect that from a certain point onwards they will take 
at most two values, one being zero and the other an integer different from zero,
provided~$\fu$ does not grow too fast. 
Considering, for example, the simpler revealing case where the components of $\fu$ 
only take the values $0$ or~$a$, where $a$ is a positive integer, 
then the numbers on the left edge are also only $0$'s or~$a$'s.

The famous case of this type of construction, where the top row is the sequence 
of prime numbers, is the subject of
the Proth-Gilbreath conjecture (see Proth~\cite{Pro1878}, Gilbreath~\cite{Gil2011},
Killgrove and Ralston~\cite{KR1959}, Odlyzko~\cite{Odl1993} and the problem sets 
of Guy~\cite[Example 12]{Guy1988},~\cite[Problem A10]{Guy2004} and 
\mbox{Montgomery~\cite[Appendix Problem 68]{Mon1994}}).
In accordance with common sense and extensive numerical observations (see~\cite{KR1959, Odl1993}), 
the conjecture states 
that all the entries of $\fw$ except $b_0$ are equal to $1$. 
\begin{conjecture}[Proth-Gilbreath]\label{ConjecturePG}
All the differences on the western edge of the~\eqref{eqPGTriangle} triangle
generated by the sequence of all primes are equal to $1$.
\end{conjecture}
However, to this day, in the sequence that is believed, 
as stated in Conjecture~\ref{ConjecturePG},
to be constantly equal to $1$, it is not even known if there are an infinite number 
of $1$'s.

In fact, the reality generated by the primes is even more intriguing 
beyond the left edge~$\fw_0$. 
Inside the~\eqref{eqPGTriangle} triangle, the phenomenon is exactly at the opposite
end. First, let us observe that on $\fw_1$, the immediately adjacent parallel line to
the left edge, there are only $0$'s or $2$'s, 
according to Conjecture~\ref{ConjecturePG}. 
Furthermore, numerical evidence shows that these values are present on this line in approximately equal proportions.
And likewise, when moving inside towards the right, on the following parallel rays
$\fw_j$, $j\ge 2$, it is also very likely that the same fact holds true
 (see Table~\ref{TableRaysP}).
\begin{conjecture}\label{ConjecturePGfull}
Let $\fw_j$, $j\ge 1$, be any line parallel to the left edge of the~\eqref{eqPGTriangle}
triangle generated by the sequence of all primes.
Denote by $\nu_d(n)$ the number of $d$'s among the first~$n$ elements of $\fw_j$.
Then, for $d\in\{0,2\}$, there exists a constant $c > 0$ 
and an integer $n_j$, such that
\begin{equation*}
    \Big|\nu_d(n)-\sdfrac{n}{2}\Big|< c \sqrt{n},\quad \text{ for $n\ge n_j$.}
\end{equation*}
\end{conjecture}
We expect the same type of uniform distribution to occur on other rays that cross
the infinite~\eqref{eqPGTriangle}, for example, those that pick the differences 
at equally spaced intervals in different directions.
Similarly, in convex domains, that are sufficiently large and are situated far enough 
from the starting row, the number of $0$'s and the number of $2$'s should be 
approximately the same with probability~one.

The structure of the discrete dynamical system generated by iteratively calculating 
neighbor gaps recorded in the~\eqref{eqPGTriangle} triangle can be better understood 
when viewed in a broader context (see~\cite{CZZ2013}) that is related to phenomena occurring 
in Pascal's triangle (see~\cite{BCZ2023, CZ2013} and Prunescu et al.~\cite{CPZ2016,  Pru2022}) or 
the patterns of numbers generated in Ducci type games (see Caragiu, 
Zaki et al.~\cite{CCZ2000, CZ2023, CZ2014, CZZ2011, CZZ2014}).
\begin{figure}[ht]
 \centering
 \mbox{
 \subfigure{
    \includegraphics[width=0.292\textwidth]{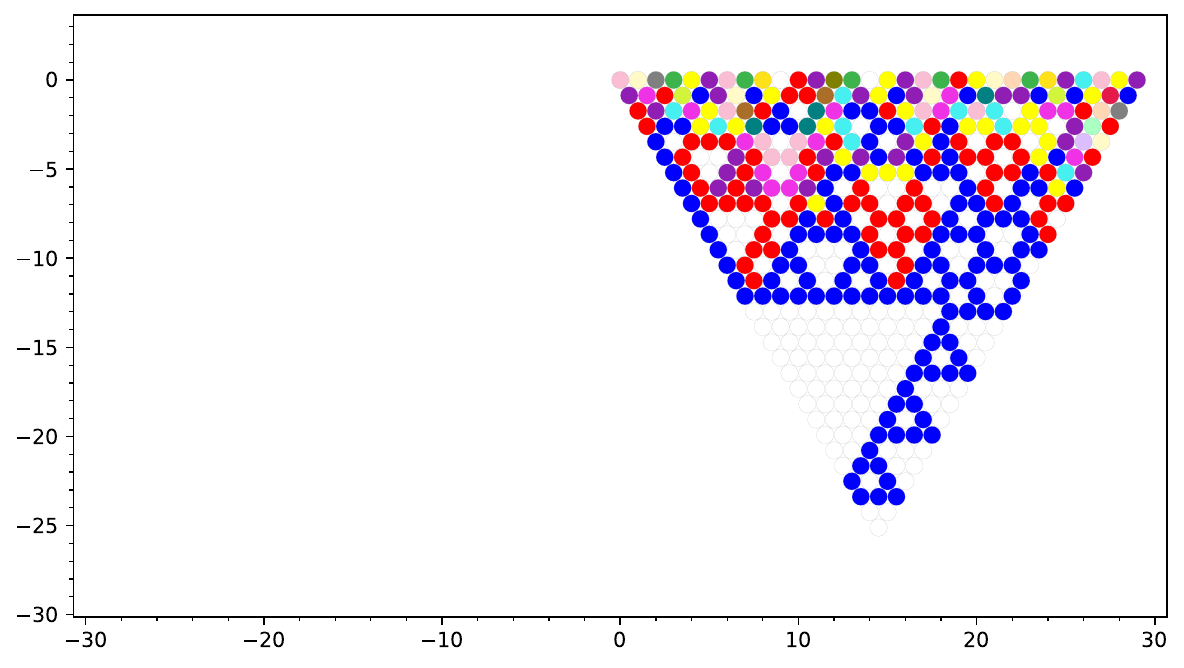}
 \label{FigureSPT1}
 }
 \subfigure{
    \includegraphics[width=0.292\textwidth]{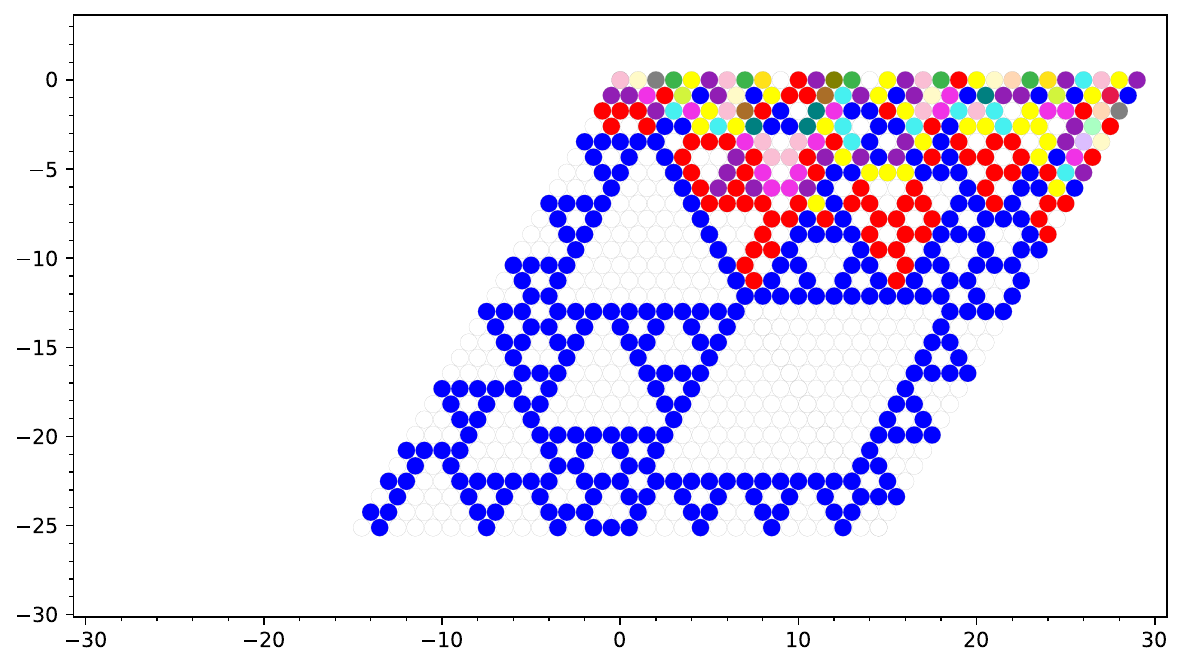}
 \label{FigureSPT2}
 }
  \subfigure{
    \includegraphics[width=0.292\textwidth]{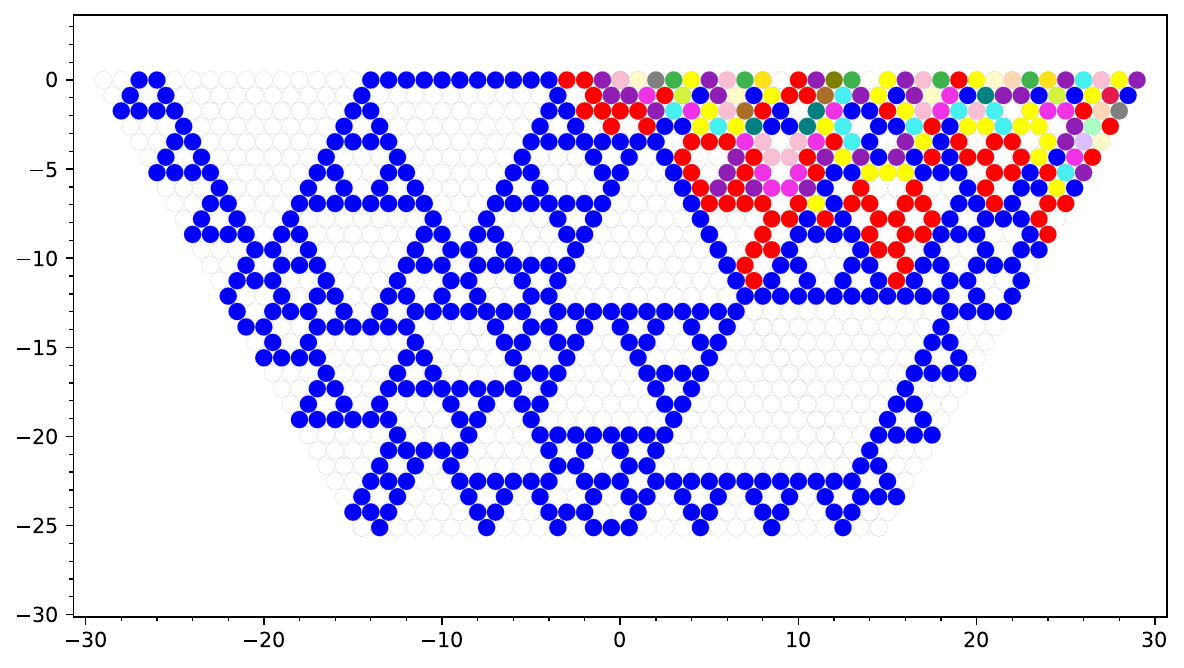}
 \label{FigureSPT3}
 }
 }
  \mbox{
  \subfigure{
    \includegraphics[width=0.30\textwidth]{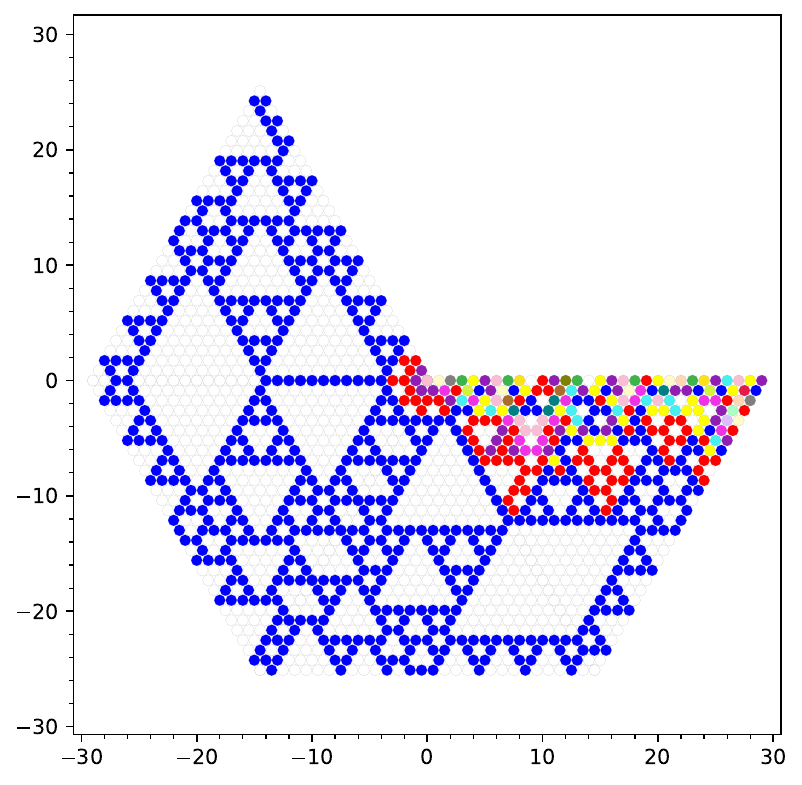}
 \label{FigureSPT4}
 }
 \subfigure{
    \includegraphics[width=0.30\textwidth]{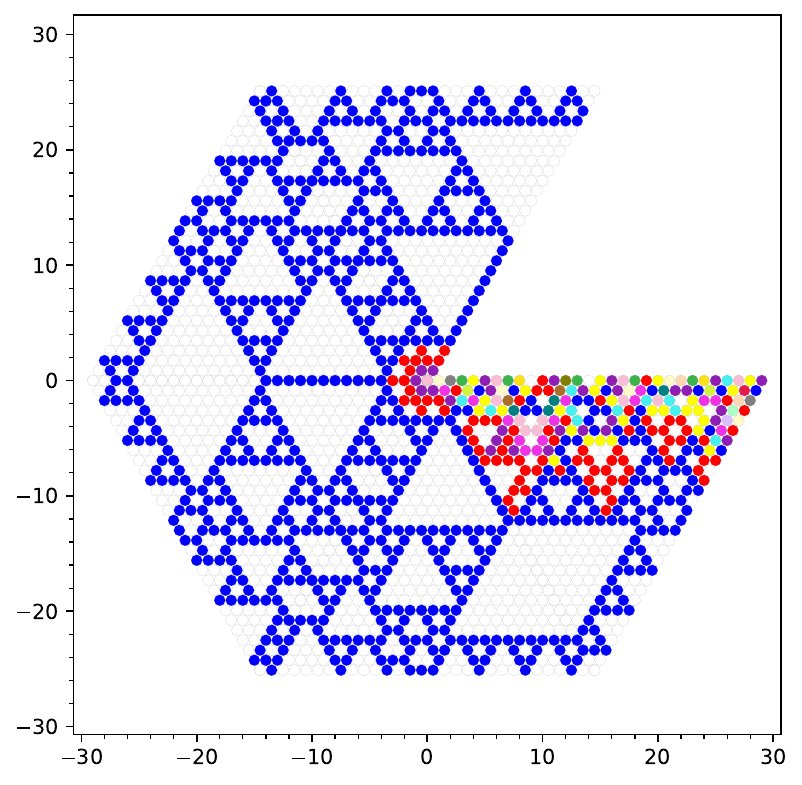}
 \label{FigureSPT5}
 }
  \subfigure{
    \includegraphics[width=0.30\textwidth]{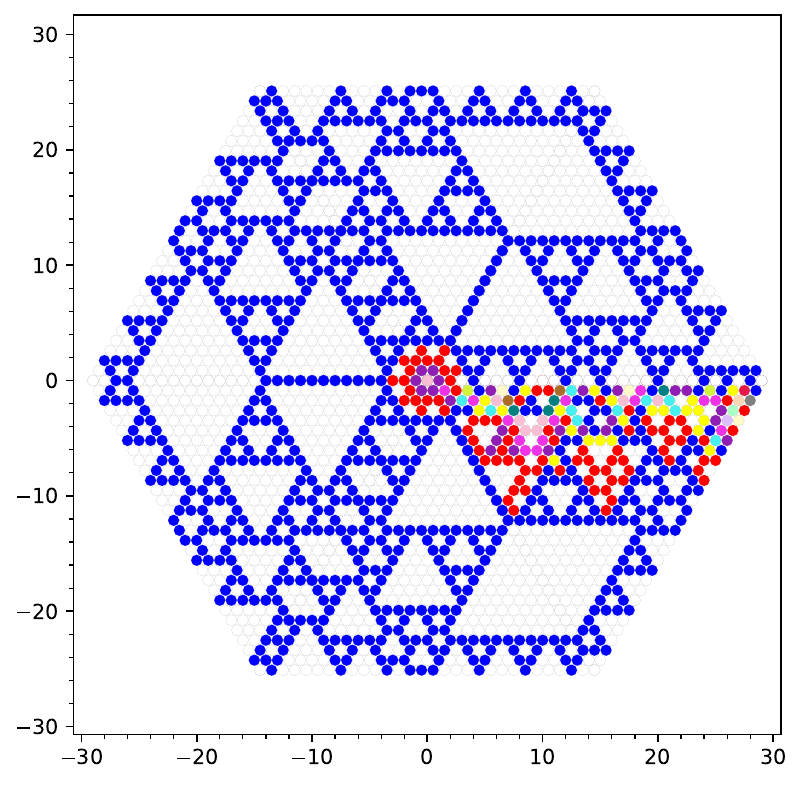}
 \label{FigureSPT6}
 }
 }
\caption{The triangle generated by $\fu$ that contains the first $30$ square-prime numbers.
(See Section~\ref{SectionSPS} for the formal definition and some appealing properties 
that the sequence of square-primes has.)
Then, with $\Upsilon^{(n)}(\fu)$ as initial rows, five more triangles are generated. 
The six figures represent the intermediate steps in forming the first layer of the helicoid.
}
 \label{FigureSPT16}
 \end{figure}

Let $\cL$ denote the set of sequences of non-negative integers, and let $\cL_0\subset\cL$ be 
the set of sequences with terms equal to $0$ or $1$, only. 
Similarly, for any integer $N\ge 0$, let $\cL(N)$ 
and $\cL_0(N)$ be the sets of finite sequences with 
$N$ non-negative integer elements, 
and sequences consisting only of $0$'s or $1$'s, respectively.
If $\fs$ is a sequence and $N\ge 0$, we denote by $\fs(N)$ 
the partial finite sequence formed by the first $N$ elements of $\fs$.

In this paper, we examine the overlying operator $\Upsilon$ that transforms the top sequence $\fu$ 
into the one on the left edge $\fw$ in the~\eqref{eqPGTriangle} triangle.
Then $\Upsilon$ is defined by
\begin{equation*}
   \begin{split}
   \Upsilon : \cL\to\cL \text{ and } \Upsilon(\fu) := \fw.
   \end{split}   
\end{equation*}
Let $\Psi : \cL\to\cL$ 
be the operator that transforms a horizontal row of the triangle into the immediately following row.
Note that the entire triangle~\eqref{eqPGTriangle} is composed of the sequence of 
successive horizontal rows $\Psi^{(j)}(\fu)$, for $j\ge 0$,
where $\Psi^{(0)}(\fu)=\fu$ is the top generating row.
Also note that the restrictions of $\Psi$ and $\Upsilon$ to $\cL_0$ have in their image 
only sequences in $\cL_0$. 
The same type of property occurs with the action of $\Psi$ and $\Upsilon$ on sequences 
$\fs$ of $0$'s and $1$'s, except for a finite number of terms. 
For these sequences, $\Psi(\fs)$ is also ultimately composed only of 
$0$'s and $1$'s, but this is not generally true for $\Upsilon(\fs)$.

The geometric correspondent of each iteration of $\Upsilon$ applied on $\fu$ results 
in the construction of a new equilateral triangle, 
rotated clockwise around the first component $a_0$ of $\fu$ by an angle of~$60$ degrees.
After six iterations of $\Upsilon$, the initial sequence $\fu$ is geometrically reached again. 
This results in the completion of the first \textit{layer} or \textit{level} of what can further 
be seen as a \textit{helicoidal discrete surface}, since in general \mbox{$\Upsilon^{(6)}(\fu)\neq \fu$ }
(see Figure~\ref{FigureSPT16} for an example of these iterations with a finite sequence of integers). 
Continuing the iterations produces a discrete \textit{helicoid} denoted $\cH=\cH(\fu)$, 
a ``Riemann surface''-like structure of non-negative integers. 
Let $\cH_n=\cH_n(\fu)$, $n\ge 1$, denote the $n$-th levels of the helicoid, so that 
\begin{equation}\label{eqHelicoid}
   \cH=\bigcup\limits_{n\ge 1}\cH_n.
\end{equation}
The first level $\cH_1$, called also the \textit{base layer}, is the union of 
six equilateral triangles with a vertex in $a_0$ and edges $\Upsilon^{(k)}(\fu)$ and $\Upsilon^{(k+1)}(\fu)$
for $k=0,\dots,5$. The subsequent levels are each generated similarly by their initial 
sequence $\Upsilon^{(6k)}(\fu)$ for $k\ge 0$.
Note that all layers are congruent geometrically, each of them covering the entire plane with numbers, 
if $\fu$ is infinite, or having the shape of a regular hexagon with $N+1$ numbers on each side 
if $\fu=(a_0,\dots,a_N)$ (see Figure~\ref{FigureHelicoid}).
\begin{figure}[th]
 \centering
    \includegraphics[width=0.99\textwidth]{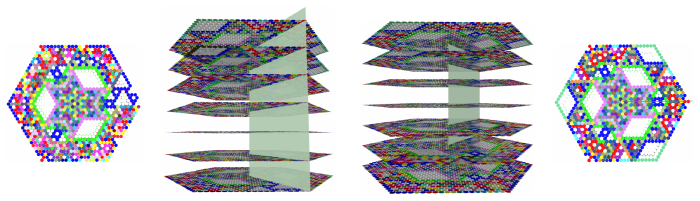}
\caption{
The first seven layers of the helicoid generated by the sequence 
$(100000, 59049, 32768, 16807, 7776, 3125, 1024, 243, 32, 1, 0, 1, 0, 0, 0, 0, 0, 1, 0, 0)$,
where the first positive integers are the first ten $5$th powers in decreasing order.
The four captures are taken in order from bottom, sides, and top. 
Distinct integers are shown in different colors. 
The helicoid has seven distinct layers, and starting from the $8$th, all 
layers coincide with the $7$th.  
The vertical strip indicates the places where the last outcome row on one layer 
transcends, becoming the first generating row on the upper layer.
}
 \label{FigureHelicoid}
 \end{figure}
The result of this process leads to a simple pattern when~$\fu$ is the sequence of prime numbers, which, at least under the assumption of the Proth-Gilbreath Conjecture,  is well understood. 
In this case \mbox{$\Upsilon^{(n}(\fu)=(2,1,1,1,\dots)$} for $n\ge 1$, so that
the interior of the corresponding generated equilateral triangles are then bounded  by 
sequences of $1$'s and have only $0$'s inside. 
\begin{figure}[th]
 \centering
 \mbox{
 \subfigure{
    \includegraphics[width=0.48\textwidth]{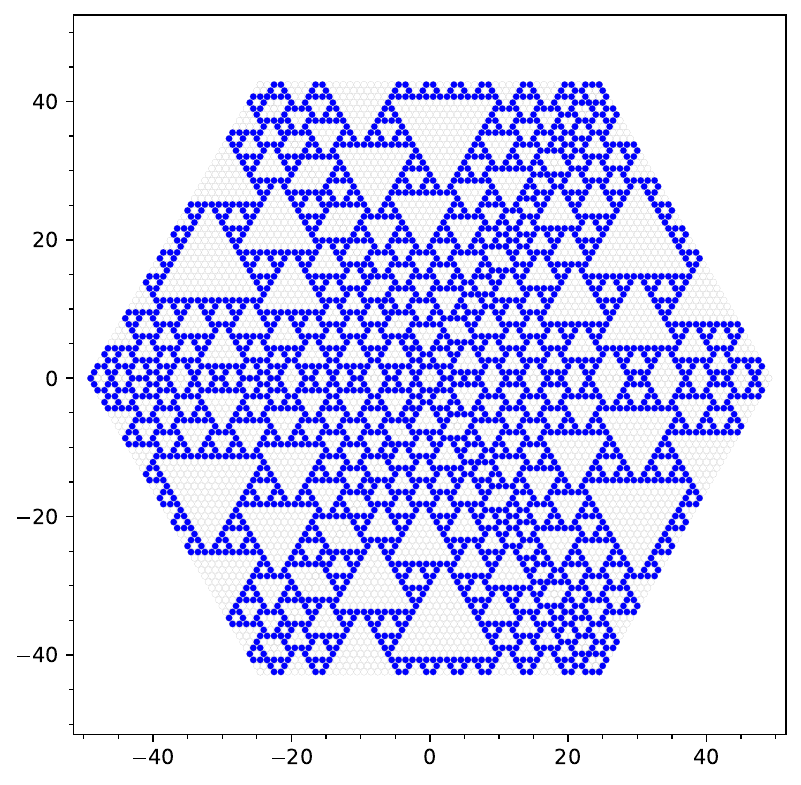}
 }
 \subfigure{
    \includegraphics[width=0.48\textwidth]{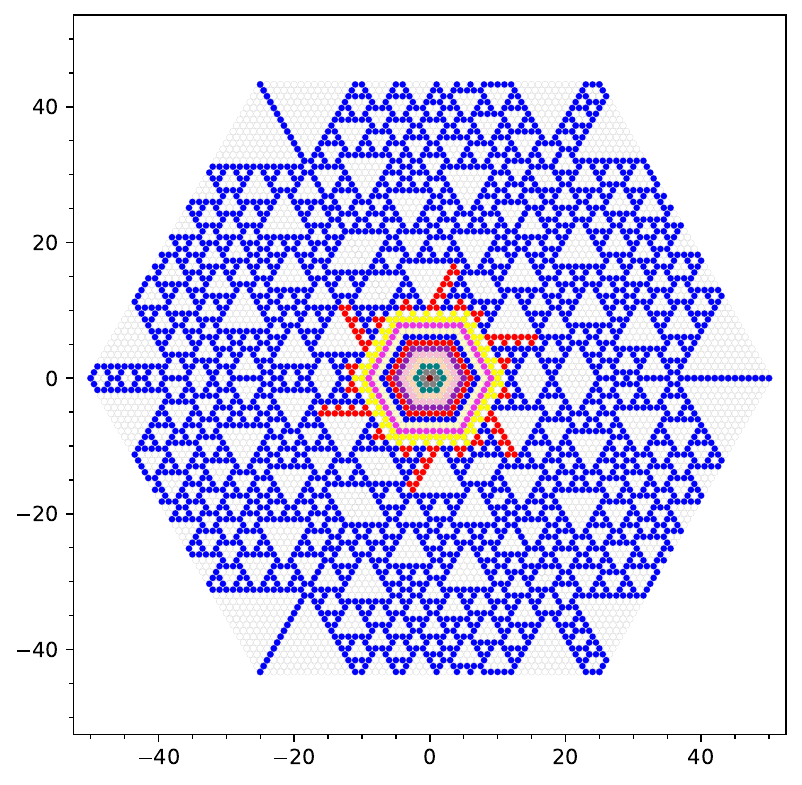}
 }
 }
\caption{
Left: The base layer of the helicoid generated by the prime-number-indicator 
function for integers in the interval $[0,50)$.
Right: The base layer of the helicoid generated by the sequence: 
$3072, 1536, 768, 384, 192, 96, 48, 24, 12, 6, 3$ (the first $10$ 
integers that are $3$ times a power of $2$ in decreasing order), 
followed by a sequence or `random bits' defined by 
$\ind_{[0,1/2)}\big(\{k\sqrt{2}\} \big)$ for $k=1,2,\dots,40$.
(Note that by Theorem~\ref{Theorem111}, in both cases, 
the helicoids have identical layers at each level.)
}
 \label{FigureTwoHexagons}
 \end{figure}
However, there is a much more interesting perspective if we mark in the initial sequence 
only the positions where prime numbers appear, their size being given 
by the rank of the terms in the sequence. 
To be precise, let $\fii=\{\ind{}_{\cP}(j) \}_{j\ge 0}$ be the indicator sequence of prime numbers, where
\begin{equation*}
   \ind{}_{\cP}(j) = 
   \begin{cases}
      1 & \text{ if $j$ is prime,}\\
      0 & \text{ else.}
   \end{cases}
\end{equation*}

Among our numerical experiments, we observed with surprise that the $6$th iteration of~$\Upsilon$ 
brings $\fii$ back to its starting point. Indeed, the helicoid generated by $\fii$ with its 
leaves composed of six equilateral triangles each, actually has only one leaf, 
all the subsequent that follow being identical to the first modulo a rotation or a reflection.

We will show that this phenomenon is not unique and that there is a large class of generating sequences 
$\fu$ that have the same property $\Upsilon^{(6)}(\fu)=\fu$.
In Theorem~\ref{Theorem111} below we consider these sequences and note, according to Theorem~\ref{TheoremAlmostAllSeq}, that their class includes a multitude of
sequences that are expected to exhibit a random behavior, like the indicator function of primes has.

\section{The main results}

The problem of demonstrating the uniform distribution of zeros and ones on 
the sides or on certain rays that cross  
the~\eqref{eqPGTriangle} triangle generated by some generic initial sequence is not within our reach.
However, beyond the specific results that can be obtained in particular cases, 
we can prove the following existence result with square-prime numbers.

\begin{figure}[ht]
 \centering
 \mbox{
 \subfigure{
    \includegraphics[width=0.48\textwidth]{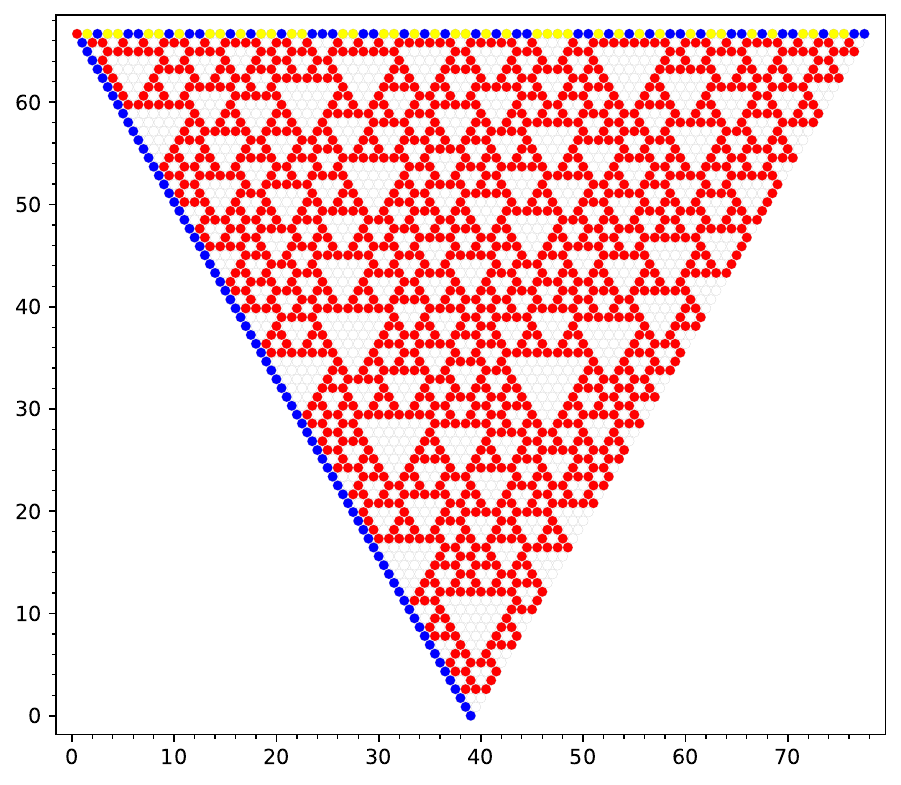}
 \label{FigLeftEdgeSPs1}
 }
 \subfigure{
    \includegraphics[width=0.48\textwidth]{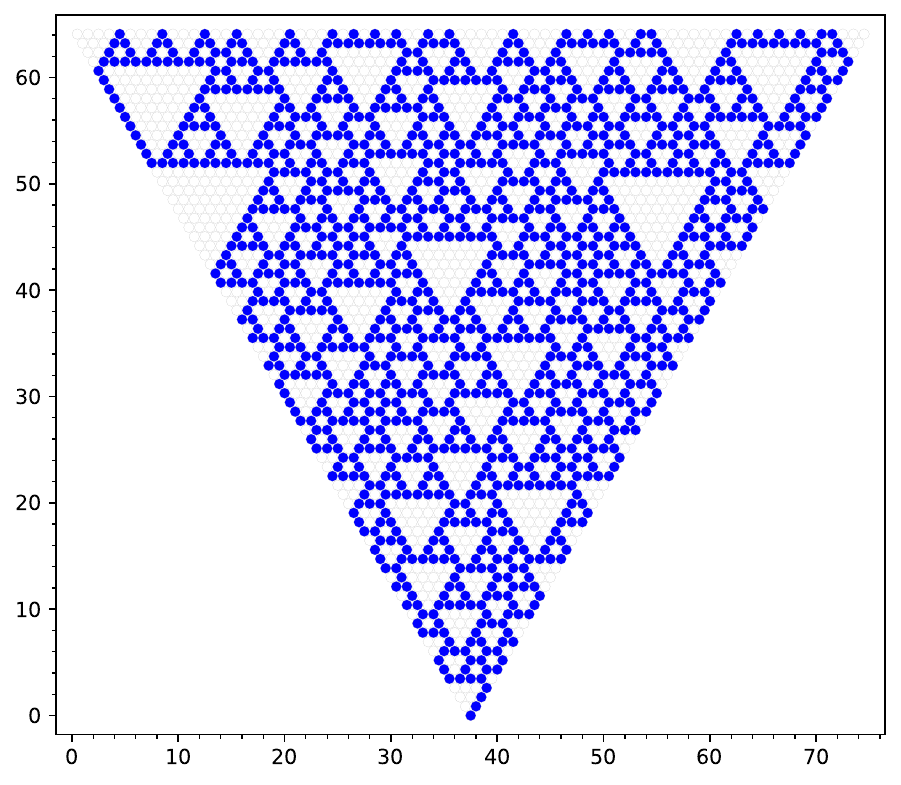}
 \label{FigLeftEdgeSPs2}
 }
 }
\caption{
Two cut-offs of the~\eqref{eqPGTriangle} triangles with integers
taken modulo $4$ (left) and modulo $2$ (right).
In the image on the left, the triangle is generated by the first~$78$ primes less than $400$, and 
in the image on the right, the top row contains the $75$ square-prime numbers less than $400$.
}
 \label{FigLeftEdgePSPs}
 \end{figure}

\begin{theorem}\label{TheoremSPGP}
There exists an infinite subsequence of square-primes $A_1 < A_2 < \ldots$ such that 
the ~\eqref{eqPGTriangle} triangle generated by $A_1, A_2, \ldots$ on the first row 
has $1$ as every other element on the left edge.
\end{theorem}
Note that Theorem~\ref{TheoremSPGP} implies that there are~\eqref{eqPGTriangle} triangles generated
by sequences of square-prime numbers for which at least half of the elements on the left edge are $1$.
 
\begin{table}[ht]
\setlength{\tabcolsep}{12pt}
\setlength{\extrarowheight}{0.5pt}
\caption{
The frequencies of the absolute values of the differences on the rays
$\fw_0,\dots,\fw_9$
that cross a cut-off of 
the~\eqref{eqPGTriangle} triangle passing parallel to its left edge.
The generating row contains the prime numbers less than one million.
Counting is done based on the values of the higher-order differences taken modulo $4$.
The notations are as follows:
$r$ is the rank of the ray, starting with the left edge $\fw$ that has rank $r=0$;
\mbox{$N$ is} the number of differences on the ray (note that there are no differences on the first row
of~\eqref{eqPGTriangle});
$z$ is the number of $0$'s, $t$ is the number of $2$'s, and $h$ is 
the number of values that are not equal to $0$ or $2$.
}
\centering
\footnotesize
\begin{tabular}{>{\scriptsize}lcccccr}
\toprule
$r$\ \ & $N$  &$z$&$t$& $z-t$ & $h$ & $(z-t)/N$ \\
\midrule
0 & 78497 & 0 & 0 & 0 & 78497 & 0.00000 \\
1 & 78496 & 39061 & 39435 & -374 & 0 & -0.00476 \\
2 & 78495 & 39272 & 39223 & 49   & 0 & 0.00062 \\
3 & 78494 & 39218 & 39275 & -57  & 1 & -0.00073 \\
4 & 78493 & 39405 & 39088 & 317  & 0 & 0.00404 \\
5 & 78492 & 39311 & 39180 & 131  & 1 & 0.00167 \\
6 & 78491 & 39030 & 39461 & -431 & 0 & -0.00549 \\
7 & 78490 & 39307 & 39182 & 125  & 1 & 0.00159 \\
8 & 78489 & 39276 & 39211 & 65   & 2 & 0.00083 \\
9 & 78488 & 39231 & 39256 & -25  & 1 & -0.00032\\
\bottomrule
\end{tabular}
\label{TableRaysP}
\end{table}
 \vspace*{11pt}
 
What is actually specific to the~\eqref{eqPGTriangle} constructions is not the singular 
fact that the edge on the left only has a special form, as it happens when the generator 
is the sequence of prime numbers and the constant shape of the edge on the left is given by 
almost an arithmetic accident. 
An impression of what happens most often can be seen in Figure~\ref{FigLeftEdgePSPs}, 
where the triangles generated by prime and square-prime numbers are placed side by side for comparison.

Then what is truly remarkable is that the rays $\fw_j$ parallel to the left edge 
tend to have a random sequence appearance. 
In particular, the sequences $\fw_j$ seem to become binary, 
and the statistics that count the number of the two types of elements indicate this. 
Thus, Table~\ref{TableRaysP} shows the closeness between the number of $0$'s and the number 
of $2$'s on rays $\fw_1,\dots,\fw_9$ for the partial~\eqref{eqPGTriangle} triangle generated by 
prime numbers less than one million and the high order differences taken modulo $4$.

To compare, the same property is observed in Table~\ref{TableRaysSP} 
when the first row begins with the square-prime numbers 
less than one million and differences taken modulo $2$. 
In this case, on each of the columns $\fw_0,\dots,\fw_9$, 
the distribution of the number of $0$'s and the number of~$1$'s is also almost evenly split in half.
And this is not only happening on the western edge, the same phenomenon occurs 
inside the~\eqref{eqPGTriangle} triangle as well. To quantify the distribution let us measure 
the rate of change on rays.

\begin{table}[ht]
\setlength{\tabcolsep}{12pt}
\setlength{\extrarowheight}{0.5pt}
\caption{
%
The frequencies of the absolute values of the differences on the rays
$\fw_0,\dots,\fw_9$
that cross a cut-off of 
the~\eqref{eqPGTriangle} triangle passing parallel to its left edge.
The generating row contains the $69179$ square-primes less than one million.
Counting is done based on the values of the higher-order differences taken modulo $2$.
The notations are as follows:
$r$ is the rank of the ray, starting with the left edge $r_0=\fw$;
\mbox{$N$ is} the number of differences on the ray (note that there are no differences on the first row
of~\eqref{eqPGTriangle});
$z$ is the number of $0$'s, $o$ is the number of $1$'s, and $h$ is 
the number of values that are not equal to $0$ or $1$.
}
\centering
\footnotesize
\begin{tabular}{>{\scriptsize}lcccccr}
\toprule
$r$ & $N$    & $z$    & $o$    & $z-o$ & $h$ & $(z-o)/N$ \\
\midrule
0 & 69178 & 34616 & 34559 & 57 & 3 & 0.00082 \\
1 & 69177 & 34684 & 34485 & 199 & 8 & 0.00288 \\
2 & 69176 & 34614 & 34556 & 58 & 6 & 0.00084 \\
3 & 69175 & 34439 & 34727 & -288 & 9 & -0.00416 \\
4 & 69174 & 34485 & 34681 & -196 & 8 & -0.00283 \\
5 & 69173 & 34808 & 34357 & 451 & 8 & 0.00652 \\
6 & 69172 & 34707 & 34458 & 249 & 7 & 0.00360 \\
7 & 69171 & 34471 & 34694 & -223 & 6 & -0.00322 \\
8 & 69170 & 34644 & 34522 & 122 & 4 & 0.00176 \\
9 & 69169 & 34689 & 34472 & 217 & 8 & 0.00314 \\
\bottomrule
\end{tabular}
\label{TableRaysSP}
\end{table}
\vspace*{11pt}

The next theorem demonstrates the existence of an underlying link between the horizontal and  
the vertical/diagonal rows of the~\eqref{eqPGTriangle} triangle.
\begin{theorem}\label{TheoremT}
   Let $\fu=(a_0,a_1,\dots)\in\cL_0$ be the first row of the~\eqref{eqPGTriangle} triangle 
and let $\fw=(b_0,b_1,\dots)$ be the sequence on its left-edge.
Let $f,g\in\FF_2[[X]]$
be the formal power series with coefficients in $\fu$ and $\fw$, respectively.
%
Let $T : \FF_2[[X]]\to\FF_2[[X]]$ be the operator 
defined by $T(f)=g$. Then:

1. The operator $T$ satisfies the following formula
\begin{equation}\label{eqOperatorT}
   \begin{split}
    \big(T(f)\big)(X) = f\Big(\sdfrac{X}{1+X}\Big)\cdot \sdfrac{1}{1+X}\,.
   \end{split}	
\end{equation}

2. The operator $T$ has the property $T^{(2)}(f)=f$ for any $f\in\FF_2[[X]]$, so that $T$ is
invertible and bijective.
\end{theorem}

The phenomenon of involution between $\fu$ and $\fw$ shown in Theorem~\ref{TheoremT} 
occurs even more exquisitely for all~\eqref{eqPGTriangle} triangles of bounded size. 
Consequently, an analogous version of Theorem~\ref{TheoremT} holds true 
for such triangles of finite size.
Here, polynomials play the role of the formal power series because they can be thought of 
as formal power series with a finite number of nonzero coefficients.

Let $R[X]^{[n]}$ denote the set of polynomials with coefficients in $R$ and degree at most $n$.
\begin{theorem}\label{TheoremTN}
   Let $N\ge 1$ be an integer and let $\fu=(a_0,a_1,\dots, a_{N-1})\in\cL_0(N)$ 
be the top row and let $\fw=(b_0,b_1,\dots,b_{N-1})$ be the sequence on the left-edge
of the~\eqref{eqPGTriangle} triangle of side $N$.
Suppose $a_0=b_0$ and let $f,g\in\FF_2[X]^{[N-2]}$, be the polynomials whose coefficients
are the components of $\fu$ and $\fw,$ respectively.
Let $T_N : \FF_2[X]^{[N-2]}\to\FF_2[X]^{[N-2]}$ 
defined by $T_N(f)=g$. Then:

1. The operator $T_N$ satisfies the following formula
\begin{equation}\label{eqOperatorTN}
   \begin{split}
    \big(T_N(f)\big)(X) \equiv f\Big(\sdfrac{X}{1+X}\Big)\cdot \sdfrac{1}{1+X}\pmod {X^N}\,.
   \end{split}	
\end{equation}

2. The operator $T_N$ has the property $T_N^{(2)}(f)=f$ for any 
$f\in\FF_2[X]^{[N-2]}$, so that $T_N$ is invertible and bijective.
\end{theorem}
\smallskip

As a follow-up of Theorems~\ref{TheoremT} and~\ref{TheoremTN},
we find that $\Upsilon^{(6)}(\fu)=\fu$ for all binary finite or infinite 
sequences $\fu$, which in particular proves  to be true 
the inceptive observation discussed at 
the end of the Introduction for the indicator function of primes.
The above theorems also imply that the helicoids generated by binary sequences 
have only one distinct layer, which is a three-petal hexagon.  
%
The next results also shows the necessity of an additional condition that must be 
fulfilled by the more general sequences in $\cL$ so that they also generate helicoids with just a single distinct layer.
The general statement includes both the case of infinite sequences and 
the case of finite sequences, as we consider the ring of polynomials embedded 
in the ring of formal power series, where polynomials have only a finite number of non-zero coefficients.

First, we define the concept of a champion in a sequence. 
We say that the term of rank~$n\ge 0$ in the sequence $\fs=\{a_k\}_{k\ge 0}$ of non-negative integers
is a \textit{champion of $\fs$}, or shortly a \textit{champion}, if $a_n>0$ and  $a_j < a_n$ 
for $0 \le j < n$.
Note that an unbounded sequence has infinitely many champions, 
and in a strictly increasing sequence all its terms are champions.
However, our point
is at the other end, at sequences with at most one champion.

\begin{theorem}\label{Theorem111}
Let $\fu\in\cL$ and let $\Upsilon : \cL \to\cL $ be the operator defined by $\Upsilon(\fu)=\fw$, 
where $\fw$  is the sequence of numbers 
obtained as the left edge of the~\eqref{eqPGTriangle} triangle generated by $\fu$. Then:
\begin{enumerate}
\setlength\itemsep{4pt}
\item
Any discrete helicoid, as defined by relation~\eqref{eqHelicoid},
which is generated by a finite 
or infinite binary sequences \mbox{$\fu\in\cL_0(N)$} 
or $\fu\in\cL_0$ has all
levels equal, that is \mbox{$\Upsilon^{(6)}(\fu)=\fu$}. 

\item
The base level of the helicoids generated by binary finite or infinite sequences are composed by three
equal diamonds rotated around the origin, and each of these 
diamonds is the union of two equilateral triangles 
that mirror each other with respect to 
the diagonal that joins the corners with obtuse vertices of the diamond. 

\item Let  $\fu\in\cL$ and suppose that $\Upsilon^{(6)}(\fu)=\fu$.
Then the sequence $\fu$  has at most one champion.
\end{enumerate}
\end{theorem}
If $\fu$ is a finite sequence, since the numbers in the helicoid $\cH(\fu)$ 
have a general tendency to decrease, in any case, all of them being 
at most equal to the largest component of $\fu$, 
it follows that the sequence $\{\cH_n(\fu)\}_{n\ge 1}$ 
of the hexagonal layers of numbers 
that compose~$\cH(\fu)$ is finally periodic. 
The sequence of these layers should then be composed of a precycle followed by a cycle, 
both of finite length. 
Note that the length of a precycle can be zero,
but the length of a cycle must always be positive. 
According to the first part of Theorem~\ref{Theorem111}, 
the length of the precycle is zero and the length of the cycle is one, for all binary generators $\fu\in\cL_0(N)$, where $N\ge 1$, as well as for infinite sequences in $\cL$.
Further investigation is needed to classify the generating sequences based on the 
length of the precycle or the cycle of the sequences of layers that their helicoids have. 
The following is an example of such a comprehensive problem.
\begin{problem}\label{Problem1}
 Let $P\ge 0$ and $C\ge 1$ be integers. Find $\fu\in\cL(N)$, for some $N\ge 1$, such that the sequence $\{\cH_n(\fu)\}_{n\ge 1}$ 
 of layers of the helicoid $\cH(\fu)$ has exactly $P+C$ distinct 
layers, grouped in a precycle
of length $P$, and followed by the endless repetition of a cycle of length~$C$.
\end{problem} 

For finite sequences with a single champion, we numerically tested various decreasing sequences
of integers followed by a sequence of $0$'s or $1$'s. 
We found that many such sequences produce helicoids 
with a sequence of layers composed by a precycle (possibly empty) and a cycle of length one.
Yet this is not generally observed, as for instance,  
the helicoid generated by the first positive $77$ seventh powers arranged 
in decreasing order and followed by 
the sequence of ten bits: $0,1,0,0,0,0,0,1,0,0$ generates a helicoid with $17$ distinct layers, 
of which $9$ are in a precycle, and $8$ are in the subsequent infinitely repeated cycle.

In Figure~\ref{FigureTwoHexagons} there are shown the base layer of two helicoids 
whose upper layers coincide with their initial ones.
In both cases the generating sequences have exactly one champion.
In Section~\ref{SectionRotationsSymmetriesHelix} we present several other examples of 
sequences with just one champion, whose helicoids have only one distinct level. 
Additionally, we show that the property of having a single champion is not sufficient to 
characterize the one-distinct-level helicoids.
Indeed, there are sequences with just one champion whose helicoids have 
more distinct levels, such as those in Figure~\ref{FigureAP3} (left) and 
Figure~\ref{FigurePowers45} (right), which have four and nine distinct levels, respectively.

For any ray $\fw_k$, for $k\ge 0$, that is parallel to the western edge $\fw_0$, 
let $R_\fw(k)$ denote the proportion of
zeros among the first $N$ components of the ray, that is,
\begin{equation}\label{eqDefRw}
    R_\fw(k) := \sdfrac{1}{N-k} 
    \#\Big\{ j : d_k^{(j)} = 1, \ 0\le j<N 
    \Big\}, \ \ \text{ for $0\le k\le N$}.
\end{equation}

Symmetrical with respect to vertical axes, let us consider the eastern edges 
of the cut-off~\eqref{eqPGTriangle}. For any fixed integer $N\ge 0$, denote by $\fe_0, \fe_1,\dots $
these edges, seen this time geometrically, in order from right to left. 
Precisely, $\fe_k=\fe_k(N)$, $0\le k\le N$, is defined by
\begin{equation*}
   \begin{split}
   \fe_k :=\big\{d_{N-k}^{(j)} : 0\le j\le N, \, j+k = N\big\}.
   \end{split}   
\end{equation*}
Then, just like for the western edges, let us denote by $R_\fe(k)$ the proportion of zeros 
on $\fe_k$:
\begin{equation}\label{eqDefRe}
    R_\fe(k) := \sdfrac{1}{N-k} 
    \#\Big\{ j : d_{N-k}^{(j)} = 1, \ 0\le j<N,\, j+k=N 
    \Big\}, \ \ \text{ for $0\le k\le N$}.
\end{equation}

The following theorem shows that almost 
all~\eqref{eqPGTriangle} triangles generated by sequences
of~$0$'s and $1$'s have nearly equal proportions of 
$0$'s and $1$'s on the rays $\fw_0,\fw_1,\dots$ and 
$\fe_0,\fe_1,\dots$
The more precise result is as follows.
%
\begin{theorem}\label{TheoremAlmostAllSeq}
For any $\varepsilon\in(0,1/2)$, there exits $\delta=\delta_{\varepsilon}>0$ and
an integer $N_{\varepsilon, \delta}$ such that, 
for any integer 
$N \geq N_{\varepsilon,\delta}$, there exists an exceptional subset
$\cE(N)\subset \cL_0(N)$ (possibly empty) 
having at most $\varepsilon\cdot \#\cL_0(N)$ elements,
such that, for any sequence $\fu(N)\in\cL_0(N)\setminus\cE(N)$, all ratios, defined by~\eqref{eqDefRw} and~\eqref{eqDefRe}, of the number of zeros on the rays
$\fw_0, \fw_1,\dots, \fw_{\lfloor  \delta N \rfloor }$ and
$\fe_0, \fe_1,\dots, \fe_{\lfloor \delta N\rfloor }$  in the corresponding triangle~\eqref{eqPGTriangle}
satisfy
\begin{equation}\label{eqTheoremAlmostAllSeq}
   R_\fw(0),R_\fw(2),\dots, R_\fw\big({\lfloor \delta N\rfloor }\big);
   R_\fe(0),R_\fe(2),\dots, R_\fe\big({\lfloor \delta N\rfloor }\big)
   \in [1/2 - \epsilon, 1/2 + \epsilon].
\end{equation}
\end{theorem}

\section{Differences with square-primes}\label{SectionSPS}
\subsection{Preliminary notes on SP-numbers}\label{SubsectionNotesOnSPS}
Merging together squares larger than one and primes into the recently introduced sequence of 
\textit{square-primes}~\cite{BMa2022, Bha2022a, Bha2022b, BCZ2023} 
proves to be a bright combination. 
Formally, the sequence is defined as the ordered union:
\begin{equation*}
   \begin{split}
   S\cP := &\bigsqcup_{k=2}^\infty \{k^2 p \mid p \text{ prime}\}\\
   = &\big \{
    8, 12, 18, 20, 27, 28, 32, 44, 45, 48, 50, 52, 63, 68, 72, 75, 76, 80, 92, 98, 99,\dots
   \big\}.
   \end{split}   
\end{equation*}
Note that there are no primes nor squares in the set $S\cP$.
Due to the uniform growth of the gaps between squares, 
this new sequence, also called \textit{SP-numbers}, 
has a type of distribution that reverberates from a distance that of the prime numbers
(for the higher order differences of primes and square-primes, in Figure~\ref{FigLeftEdgePSPs} 
there are side-by-side two triangles generated by them for comparison).
However, the change in the arithmetic nature from primes to composite numbers, 
makes proving some remarkable properties that the prime number sequence has 
transferred to SP-numbers not as difficult 
if we employ what we already know about prime numbers.

Let $s_n$, $n\ge 1$, denote the $n$th square-prime. A few such numbers are $s_1=8$, $s_{21}=99$,
$s_{76}=404$ and $s_{1000}=7900$.
An asymptotic estimate (see~\cite[Theorem 4.1]{Bha2022a}) 
shows that $s_n$ and $p_n$ have a similar order of magnitude:
\begin{equation*}
   \begin{split}
   s_n = \big(\zeta(2) - 1\big)\cdot \sdfrac{n}{\log n} + O\left(\sdfrac{n}{\log n^2}\right),
   \end{split}   
\end{equation*}
and $s_n<p_n$ for large $n$ because $\zeta(2) - 1\approx 0.64493 < 1$.
An analogous result to Dirichlet's Theorem~\cite[Theorem 6.1]{Bha2022a} 
on primes in arithmetic progressions holds
with a different constant depending on the progression for square-primes, also.

In the increasing sequence of square-primes, \textit{twins} are pairs of neighbor numbers that are~$1$ distance apart (such as $27$ and $28$ or $44$ and $45$), and we know that there are 
infinitely many such pairs~\mbox{\cite[Theorem 4.3]{Bha2022a}.}
Closely related to this is the following lemma that we need in the proof of Theorem~\ref{TheoremSPGP}.
The proof of the lemma appears in~\cite{BMa2022} and for the sake of completeness, we include here, also.

\begin{lemma}[\!\!{\cite[Theorem 2.1]{BMa2022}}]\label{LemmaSP2}
For any positive integer $x$, there exist two square-prime numbers $a,b$ such that $x = a-b$. 
\end{lemma}
\begin{proof}
We partition the set of all positive integers into the following five subsets:
\begin{enumerate}
\item[(\texttt{i})] $S_1=\{1\}$.
\item[(\texttt{ii})] $S_2=\cP$, the set of all primes.
\item[(\texttt{iii})] $S_3=\{x : x \not\in\cP,\, 2\nmid x\}$, the set of all odd composite numbers.
\item[(\texttt{iv})] $S_4=\{x : x = 2p_1\cdots p_k,\, k\ge 1, \text{ for some distinct odd primes}\}$, 
the set of all even composite square-free numbers.
\item[(\texttt{v})] $S_5=\{x : x = p^2d, p\in\cP, d\ge 1\}$, 
the set of non-square-free numbers.
\end{enumerate}

\smallskip 
\noindent
Note that $\NN\setminus\{0\}=S_1\cup S_2\cup S_3\cup S_4\cup S_5$, hence it suffices to prove 
the existence of a pair of $SP$-numbers with difference $x$, separately, for $x\in S_j$, $1\le j\le 5$.

\medskip
\noindent (\texttt{i}) 
We have only one candidate, $x=1$, in this case.
There exist infinitely many neighbor SP-numbers~\cite[Theorem 4.3]{Bha2022a}, 
but it is enough to consider the first example $(a,b)=(28,27)$, for which $a-b=1$, 
which proves the case.

\medskip
\noindent (\texttt{ii})
%
Let $x\in\cP$ be fixed, and let also $p\in \cP$ be a different prime number.
Consider the following Pell equation in variables $m,n\in\ZZ$:
\begin{equation}\label{eqPell}
   \begin{split}
  m^2 - pxn^2 = 1.  
   \end{split}	
\end{equation}
Since $px$ is not a square, we know after Pell that~\eqref{eqPell} has at least a solution with
$m,n>1$. Let $(M,N)$ be such a solution. Then,
$M^2 - pxN^2 = 1$, and multiplying this equality by $x$, we find that
\begin{equation}\label{eqrefSPcase2}
   xM^2 -p(xN)^2 = x. 
\end{equation}
Observing that $a:=xM^2$ and $b:=p(xN)^2$ are both SP-numbers, the equality~\eqref{eqrefSPcase2},
which becomes $a-b=x$, proves the lemma in this case.

\medskip
\noindent (\texttt{iii})
Suppose now that $x$ is an odd composite square-free number.
Then, $x = py$, where $p\ge 3$ is prime and $y\ge 5$ is prime or a product of distinct primes $\ge 5$.
Thus $y$ is also necessarily odd. Let $y=2K+1$ for some integer $K\ge 2$.
It then follows that $y$ is a difference of two squares:
$y=(K+1)^2-K^2$, which implies
\begin{equation}\label{eqxisaminusb1}
   x = p\big((K+1)^2-K^2\big) = p(K+1)^2-pK^2.
\end{equation}
Let $a:= p(K+1)^2$ and $b:=pK^2$.
Since $K>1$, both $a$ and $b$ are square-prime numbers and~\eqref{eqxisaminusb1} shows that
$x=a-b$, which concludes the proof in this case as well.

\medskip
\noindent (\texttt{iv})
Now let us assume that $x$ is an even composite square-free number.
Then  $x = 2y$, where $y\ge 3$ is prime or a product of distinct primes $\ge 3$.
Reasoning as in the previous case, we find that $y=2K+1$, and the analogue of~\eqref{eqxisaminusb1} is 
\begin{equation}\label{eqxisaminusb2}
   x = 2\big((K+1)^2-K^2\big) = 2(K+1)^2-2K^2,
\end{equation}
where $K$ is a positive integer that can also be equal to $1$ this time.

Let $a:= 2(K+1)^2$ and $b:=2K^2$. Note that if $K>1$, then both $a$  and $b$ are SP-numbers
and~\eqref{eqxisaminusb2} shows that $x=a-b$.  

In the remaining possibility when $K=1$, we have $y=3$, so $x=6$, 
which can also be written as a difference of square-primes:
$6=2\cdot 3^2-3\cdot 2^2$,
concluding the argument in case~(\texttt{iv}).

\medskip
\noindent (\texttt{v})
Let us assume now that $x>1$ is not square-free, that is,  
$x=c^2y$ for some integers $c>1$,~$y\ge 1$, and $y$ is square-free.
Then, from the proved cases (\texttt{i})-(\texttt{iv}), we know that there exist two square-prime
numbers $a'$ and $b'$ such that $y=a'-b'$.
Let us say that $a' = p's^2$ and $b' = p''t^2$, where $p', p''$ are prime numbers 
and $s, t > 1$ are integers.
These yield:
\begin{equation*}
   x= c^2y = c^2(a'-b') = p's^2c^2-p''t^2c^2.   
\end{equation*}
Let $a := p'(sc)^2$ and $b := p''(tc)^2$. Since $p', p''$ are primes and $sc,tc>1$, both $a$ and $b$
are SP-numbers, and the above shows that $x=a-b$. 
This concludes the proof for case (\texttt{v}) and also the entire proof of the lemma.
\end{proof}
On combining Lemma~\ref{LemmaSP2} with the fact that any distance between
two square-primes is replicated infinitely often as the difference between other square-primes~\cite[Theorem 4.3]{Bha2022a}, we find that all positive integers appear 
infinitely often as differences between square-primes.
\begin{lemma}\label{LemmaInfiniteGaps}
   Any positive integer appears infinitely often as a difference between square-primes.
\end{lemma}

\subsection{Proof of Theorem~\ref{TheoremSPGP}}\label{SubSectionProofTheoremSPS}
We begin by proving a related result which shows that a~\eqref{eqPGTriangle} triangle 
with `controlled size' elements on the eastern edge can be enlarged by padding it in such a way 
that the new southern vertex has a preset number $Z$.
%
\begin{proposition}\label{PropositionBorderZ}
Consider a~\eqref{eqPGTriangle} triangle with integers 
$0\le B_1\le B_2\le\cdots\le B_m$ on 
the eastern edge $\fe_1$. 
Then, there exists an integer $C_1\ge B_1$, such that 
the triangle bordered by a new eastern edge $\fe_0$ obtained by calculating the 
differences generated by the addition of $C_1$ at the end of the generating row 
has components $C_1, C_2,\dots, C_m,C_{m+1}$ with $C_j\ge B_j$ for $1\le j\le m$.
Moreover, given an integer $Z\ge 0$,
we can choose $C_1$ such that $C_{m+1} = Z$.
\end{proposition}
\begin{proof}
With some arbitrary integers $A_1,\dots, A_{m-1}$, 
the triangle in the statement of the proposition is as follows:
\begin{equation*}
\setlength{\extrarowheight}{-0.25pt}
\setlength{\extrarowheight}{1mm}
   \begin{tabular}{ccccccccccccccccccccc}
 & $A_1$ && $A_2$ && $\dots$ && $A_{m-1}$ && $B_1$ && $C_1$ &\\[2mm]
& &  $\dots$   & &  $\dots$ &&  $\dots$ &&  $B_2$ &&  $C_2$\\[2mm]
  &   && $\dots$  && $\dots$ && $B_3$ && $C_3$ &\\[2mm]  
   & &   &&  $\dots$ &&  $\dots$ && $\dots$   &  \\[2mm]
     &   &&  && $B_m$  && $C_m$ &  & \\[2mm]
     &   &&  && & $Z$  && && &  & \\[2mm]
  \end{tabular}
\end{equation*}
%
We proceed backwards, from bottom to top. Let us assume that $Z\ge 0$ is given and it takes
the position of $C_{m+1}$.
Then, according to the definition, we may take $C_{m}$ such that $C_m-B_m=Z=C_{m+1}$, that is,
$C_m=C_{m+1}+B_m\ge B_m$.

Next, take $C_{m-1}$ such that $C_{m-1}-B_{m-1}=C_m$, so that $C_{m-1}=C_m+B_{m-1}\ge B_{m-1}$.

Likewise, inductively, it follows that we may take $C_{1}$ such that $C_1-B_1=C_2$, so that
$C_1=C_2+B_1\ge B_1$.

In conclusion, we obtained $C_1$ and the sequence $C_1,\dots,C_{m+1}$, 
which satisfies the inequalities $C_1\ge B_1, \dots, C_m\ge B_m$, and additionally $C_{m+1}=Z$,
thus proving the lemma.
\end{proof}

\begin{remark}\label{RemarkL1}
Let us note that the proof of Proposition~\ref{PropositionBorderZ} 
also allows the assumption of a different preset 
order between $C_j$ and $B_j$ for $1\le j\le m$.
Indeed, starting in the same way from $Z$ and recursively calculating in reverse order the elements 
$C_j$ from the new equalities given by the preset order, we obtain $C_1$, 
the new element of the first row, which ensures that the southern vertex 
of the~\eqref{eqPGTriangle} triangle is $Z$.
\end{remark}

Numerical experiments show that square-prime numbers are very handy 
for generating~\eqref{eqPGTriangle} triangles that have various properties. 
For example, one that has alternately $1$ on the western edge is:
{\small
\begin{equation}\label{eqAlternate1s}
\setlength{\tabcolsep}{2.5pt}
\setlength{\extrarowheight}{0.5pt}
\begin{tabular}{cccccccccccccccccccccccccccc}
 & 27 && 28 && 44 && 76 && 98 && 112 && 153 && 171 && 180 && 188 && 292 && 316 &\\
 && 1 && 16 && 32 && 22 && 14 && 41 && 18 && 9 && 8 && 104 && 24 &&\\
 &&& 15 && 16 && 10 && 8 && 27 && 23 && 9 && 1 && 96 && 80 &&&\\
 &&&& 1 && 6 && 2 && 19 && 4 && 14 && 8 && 95 && 16 &&&&\\
 &&&&& 5 && 4 && 17 && 15 && 10 && 6 && 87 && 79 &&&&&\\
 &&&&&& 1 && 13 && 2 && 5 && 4 && 81 && 8 &&&&&&\\
 &&&&&&& 12 && 11 && 3 && 1 && 77 && 73 &&&&&&&\\
 &&&&&&&& 1 && 8 && 2 && 76 && 4 &&&&&&&&\\
 &&&&&&&&& 7 && 6 && 74 && 72 &&&&&&&&&\\
 &&&&&&&&&& 1 && 68 && 2 &&&&&&&&&&\\
 &&&&&&&&&&& 67 && 66 &&&&&&&&&&&\\
 &&&&&&&&&&&& 1 &&&&&&&&&&&&
  \end{tabular}
\end{equation}
}

Turning now to the proof of Theorem~\ref{TheoremSPGP}, 
let us suppose that $m\ge 2$ is even and the~\eqref{eqPGTriangle} triangle generated 
by $\fu=(A_1< A_2<\cdots<A_m)$
is given such that it satisfies the hypothesis of the theorem. 
For the case where there are only square-prime numbers on the first row 
and the elements on the western edge have $1$  as every other element,  
there are many small triangles satisfying these requirements 
(see the numerical triangle~\eqref{eqAlternate1s} that has inserted into it a few such examples).
Our objective is to border the triangle of size $m$ with two additional edges to the east in such a way that 
the larger triangle, with a side length of~$m+2$, also satisfies the requirements 
of Theorem~\ref{TheoremSPGP}. 
Then, by induction, we will conclude that the result holds for any triangle of even size~$m \ge 2$.

Denote by $\fe''=(A_m,D_1,\dots, D_{m-1})$ the eastern edge of the given triangle of size $m$.
Let~$X$ and $Y$, be the two new numbers that will continue $\fu$, and let us denote 
the bordering edges they generate by $\fe'=(X,E_1,\dots,E_m-1,E_m)$ and 
$\fe=(Y,F_1,\dots,F_{m-1},F_m,F_{m+1})$, respectively.
Let $Z$ be the integer on the southern vertex, that is, in the previous notation, $F_{m+1}=Z$, 
and in the hypothesis of the theorem~$Z=1$. All these notations can be seen at a glance in the 
following display:
%
\begin{equation*}
\setlength{\tabcolsep}{2.5pt}
\setlength{\extrarowheight}{-0.25pt}
   \begin{tabular}{ccccccccccccccccccccc}
 & $A_1$ && $A_2$ && $\dots$ && $A_m$ && $X$ && $Y$ &\\[2mm]
& &  $A_2-A_1$   & &  $\dots$ &&  $D_1$ &&  $E_1$ &&  $F_1$\\[2mm]
  &   && $\dots$  && $D_2$ && $E_2$ && $F_2$ &\\[2mm]  
  &   && \multicolumn{6}{c}{\ \ \ \dotfill\ \ } &\\[2mm] 
   & &   &&  $D_{m-1}$ &&  $E_{m-1}$ && $F_{m-1}$   &  \\[2mm]
     &   &&  && $E_{m}$  && $F_{m}$ &  & \\[2mm]
     &   &&  && & $F_{m+1}=Z$  && && &  & \\[2mm]
  \end{tabular}
\end{equation*}
%
Let $\Delta$ denote the difference $\Delta= Y-X$.
Proceeding backwards as in the proof of Proposition~\ref{PropositionBorderZ}, we find out the conditions that
$X$ and $\Delta$ must meet. This time, we start with the necessary conditions and verify that 
they actually fulfill their role.
The two conditions are:
\begin{equation}\label{Condition1}
   X \ge A_m+Z+D_1+D_2+\cdots+D_{m-1}\phantom{Y-Z}
\end{equation}
and 
\begin{equation}\label{Condition2}
   \Delta  = Y-X = Z+\lOddB D_1+D_3+\cdots+D_{m-1}\rOddB,
\end{equation}
%
where, the double square brackets indicate that only the $D_j$'s with odd indices are added.

Let us first note that both conditions~\eqref{Condition1} and~\eqref{Condition2} 
depend only on the eastern edge $\fe''$ 
of the initial triangle that we border and the integer $Z$ that 
we want as the southern vertex of the bordered triangle.
Then we know from Lemma~\ref{LemmaInfiniteGaps} that there are pairs of square-primes, 
no matter how big and how many, that fulfill them.

Taking into account condition~\eqref{Condition1}, the numbers on the first added border layer $\fe'$ are:
\begin{equation*}
   \begin{split}
   E_1 &= X - A_m,\\
   E_2 &= E_1-D_1 = X - A_m -D_1,\\
   E_3 &= E_2-D_2 = X - A_m -D_1-D_2,\\
   &\cdots\\
   E_{m-1} &= E_{m-2}-D_{m-2} = X-A_m - D_1-D_2-\cdots-D_{m-2},\\
   E_{m} &= E_{m-1}-D_{m-1} = X-A_m - D_1-D_2-\cdots-D_{m-2}-D_{m-1}.
   \end{split}   
\end{equation*}
Further, employing condition~\eqref{Condition2} as well, these show that the differences on  $\fe$, 
the outer layer of the border, are
\begin{equation*}
   \begin{split}
   F_1 &= Y-X = Z+\lOddB D_1+D_3+\cdots+D_{m-1}\rOddB, \\
   F_2 &= E_1 -F_1 = X - A_m -Z - \lOddB D_1 + D_3+\cdots+D_{m-1}\rOddB,\\
   F_3 &= E_2 -F_2 = Z +  \lOddB D_3+\cdots+D_{m-1}\rOddB,\\
   F_4 &= E_3 -F_3 = X - A_m -Z - D_1 - D_2 - \lOddB D_3+\cdots+D_{m-1}\rOddB,\\
   F_5 &= E_4 -F_4 = Z +  \lOddB D_5+\cdots+D_{m-1}\rOddB,\\
   F_6 &= E_5 -F_5 = X - A_m -Z - D_1 - D_2 - D_3 -D_4 -\lOddB D_5+\cdots+D_{m-1}\rOddB,\\
   &\cdots\\
   F_{m-1} &= E_{m-2} -F_{m-2} = Z +  \lOddB D_{m-1}\rOddB,\\
   F_m &= E_{m-1} -F_{m-1} = X - A_m -Z - D_1 - D_2 - \cdots -D_{m-2} -\lOddB D_{m-1}\rOddB,\\
   F_{m+1} &= E_{m} -F_{m} = Z.
   \end{split}   
\end{equation*}

In conclusion, by extending the initial sequence $\fu$ with the two square-prime numbers $X$ and $Y$, 
we obtained a triangle whose southern vertex is $Z=1$, as desired, 
which concludes the proof of Theorem~\ref{TheoremSPGP}.

\smallskip
It is worth noting that in the above proof, the actual value of the number $Z$ 
did not play any special role, as $Z$ could have taken any value $Z\ge 0$. 
Thus, the following more general result holds. 
Given a sequence of non-negative integers $\{w_j\}_{j\ge 1}$, there exists a~\eqref{eqPGTriangle}
triangle generated by an increasing sequence of square-primes whose western edge is
$(*,w_1,*,w_2,*,w_3,*,\dots)$, where the stars are some unspecified non-negative integers.

\begin{theorem}\label{TheoremSPGP2}
Let $\bw=\{w_j\}_{j\ge 1}$ be a sequence of non-negative integers.
Then, there exists an increasing sequence of square-primes such that the~\eqref{eqPGTriangle} triangle 
they generate has on the western edge a sequence whose even-indexed elements are the elements 
of $\bw$.

\end{theorem}

\section{\texorpdfstring{\eqref{eqPGTriangle}}{P-G} triangles in the mirror-rays}

Suppose all entries on the top row of the~\eqref{eqPGTriangle} are from $\cL_0$.
Then the entire triangle has the same property, also.

Next we show that there is a close, analytically expressible link between 
the top row~$\fu$ and the left edge $\fw$ of the triangle.
This link actually extends across the entire triangle, 
because if one cuts a few rows from the top,
then the link shall be maintained between the remaining new-first row 
and the remaining elements on the new-left edge.
On the other hand, cutting also a few vertical (geometrically rather oblique) columns on the left side, 
all parallel to edge $\fw$, then the link becomes one between the first remaining row and the first remaining leftmost edge.

Denote the components of the top sequence by $\fu=(a_0,a_1,\dots)\in\cL_0$, and
the componets of the left-edge sequence by $\fw=(b_0,b_1,\dots)\in\cL_0$.
Their corresponding formal power series are 
$f=f(X)\in\FF_2[[X]]$ and $g=g(X)\in\FF_2[[X]]$, respectively, where
\begin{equation}\label{eqfg}
   f(X) = \sum_{n\ge 0}a_n X^n\ \ \text{ and }\ \ 
   g(X) = \sum_{n\ge 0}b_n X^n .
\end{equation} 
Suppose $a_0=b_0$ is the element in the upper left corner of~\eqref{eqPGTriangle}, 
and that it is also the constant term of both series $f$ and $g$.

\subsection{Lemmas}\label{SubsectionLemmas}
The first lemma shows that 
starting somewhere on the right side of the Pascal triangle and adding
the numbers placed on the following rows diagonally to the left, 
the partial sums that we obtain are equal to the first number on the right 
of the next row.

\begin{lemma}\label{LemmaPascalR}
For any integers $K,n\ge 0$, we have
\begin{equation}\label{eqPascalR}
        \mbinom{K}{K} + \mbinom{K+1}{K}+\cdots +
        \mbinom{K+n}{K}= \mbinom{K+n+1}{K+1}.
\end{equation}
\end{lemma}
\begin{proof}
The proof follows by induction over $n$ using the recursive formula that generates Pascal's triangle.

If $n=1$, we have
$\mbinom[0.7]{K}{K}+\mbinom[0.7]{K+1}{K}=1+(K+1)=K+2=\mbinom[0.7]{K+2}{K+1}$.

Suppose~\eqref{eqPascalR} holds for some $n\ge 1$. Then, since
\begin{equation*}
   \begin{split}
      \mbinom{K+n+1}{K+1}+\mbinom{K+n+1}{K} =  \mbinom{K+n+2}{K+1}
   \end{split}   
\end{equation*}
it follows that~\eqref{eqPascalR} also holds for $n+1$, which concludes the proof of the lemma.
\end{proof}

The next lemma gives the formal power series expression and the rational representation 
of a power of the fundamental series $F(X)=1/(1+X)$, 
whose coefficients are all equal to~$1$ in $\FF_2$.
\begin{lemma}\label{LemmaPowerF}
Let $F\in\FF_2[[X]]$, $F(X)= 1+X+X^2+\cdots$.
Then, for any integer $N\ge 0$, we have
\begin{equation}\label{eqFtok}
    \big(F(X)\big)^{N+1}=\sum_{n\ge 0}
        \mbinom{N+n}{N} X^n = \mfrac[1.3]{1}{(1+X)^{N+1}}.
\end{equation}
\end{lemma}
\begin{proof}
The proof is by induction.
If $N=1$, relation~\eqref{eqFtok} holds because 
$F(X)\cdot F(X) = \sum_{n\ge 0}(n+1)X^n$,
the coefficients being given by the equality
\begin{equation*}
\sum_{r\ge 0}\sum_{\substack{s\ge 0\\ r+s=n}}1\cdot 1=n+1.    
\end{equation*}
Let $K\ge 1$ and suppose that the coefficient of $X^s$ in the power series of $\big(F(X)\big)^{K}$ is
$\mbinom[0.7]{K-1+s}{K-1}$ for all $s\ge 0$. Then, the coefficient of $X^n$ in the product
 $F(X)\cdot F^K(X)$ is
\begin{equation*}
   \begin{split}
   \sum_{r\ge 0}\sum_{\substack{s\ge 0\\ r+s=n}}1\cdot \mbinom{K-1+s}{K-1} 
   = \mbinom{K-1+0}{K-1}+\mbinom{K-1+1}{K-1}+\cdots +\mbinom{K-1+n}{K-1}
   =\mbinom{K+n}{K},
   \end{split}   
\end{equation*}
where the last equality follows from Lemma~\ref{LemmaPascalR}.
Since these are exactly the coefficients of $\big(F(X)\big)^{K+1}$, this concludes the proof of the lemma.
\end{proof}

\subsection{Proof of Theorem~\ref{TheoremT}}
1. Let $\fw=(b_0,b_1,b_2,\dots)$ be the left edge of~\eqref{eqPGTriangle} generated by $\fu$.
Then $b_0 = a_0$, $b_1=a_0+a_1$, $b_2=a_0+2a_1+a_2$ and so on.
Then, by induction, one finds that the general formula 
for~$b_n$ is
\begin{equation}\label{eqbn}
    b_n = \binom{n}{0}a_0+\binom{n}{1}a_1+\cdots
    +\binom{n}{n}a_n\,.
\end{equation}
The formal power series of $\big(T(f)\big)(X)$ is 
$b_0+b_1X+b_2X^2+\cdots$, and to deduce a functional expression
for it, we rearrange the terms using formula~\eqref{eqbn}.
Collecting together similar terms with the same coefficient $a_n$, we see that
\begin{equation*}
  \begin{split}
    \big(T(f)\big)(X)
    = & a_0\big(1+X+X^2+X^3+\cdots\big)\\
     &+a_1X\bigg(\sdbinom{1+0}{1} +\sdbinom{1+1}{1}X 
     +\sdbinom{1+2}{1}X^2 
     +\sdbinom{1+3}{1}X^3+\cdots\bigg)\\
     &+a_2X^2\bigg(\sdbinom{2+0}{2} +\sdbinom{2+1}{2}X 
     +\sdbinom{2+2}{2}X^2 
     +\sdbinom{2+3}{2}X^3+\cdots\bigg)\\
     & + \cdots\\
    &+a_nX^n\bigg(\sdbinom{n+0}{n} +\sdbinom{n+1}{n}X 
     +\sdbinom{n+2}{n}X^2 
     +\sdbinom{n+3}{n}X^3+\cdots\bigg)\\
     &+\cdots
    \end{split}  
\end{equation*}
Using Lemma~\ref{LemmaPowerF} on each of the lines of the relation above we obtain
\begin{equation*}
  \begin{split}
    \big(T(f)\big)(X)
    & =\sdfrac{a_0}{1+X} +\sdfrac{a_1X}{(1+X)^2}
    +\sdfrac{a_2X^2}{(1+X)^3} +\cdots
    +\sdfrac{a_nX^n}{(1+X)^{n+1}} +\cdots
    \\
    & = \sdfrac{1}{1+X}\Big(a_0+a_1\sdfrac{X}{1+X}
          +a_2\bigg(\sdfrac{X}{1+X}\bigg)^2
          +\cdots
          +a_n\bigg(\sdfrac{X}{1+X}\bigg)^n
          +\cdots
          \Big)\\
    & = \sdfrac{1}{1+X} f\Big(\sdfrac{X}{1+X}\Big)\,,
    \end{split}  
\end{equation*}
which proves the first point of the theorem.

2. We apply formula~\eqref{eqOperatorT} twice. 
First we obtain 
\begin{equation*}
  \begin{split}
   T^{(2)}\big(f(X)\big) 
   = T\Big(T\big(f(X)\big)\Big)
    = \sdfrac{1}{1+X} T(f)\Big(\sdfrac{X}{1+X} \Big),
  \end{split}  
\end{equation*} 
and then, continuing, on the second application, after reducing the terms in the rational factions of $\FF_2(X)$ and making the necessary cancellations, we obtain:
\begin{equation*}
  \begin{split}
   T^{(2)}\big(f(X)\big) 
    = \sdfrac{1}{1+X}\cdot \sdfrac{1}{1+\frac{X}{1+X}}f\bigg(\sdfrac{\frac{X}{1+X}}{1+\frac{X}{1+X}}\bigg) = f(X).
  \end{split}  
\end{equation*} 
It then follows that $T$ is invertible, so that is bijective and $T^{-1}=T$.
This concludes the proof of the theorem.

\subsection{Proof of Theorem~\ref{TheoremTN}}
For any integer $N\ge 0$, the set of finite sequences $\cL_0(N)$ is in one-to-one correspondence with
the set of polynomials $\FF_2[X]$, which
in turn is embedded in $\FF_2[[X]]$, viewing the polynomials as formal power series 
with only finitely many non-zero coefficients. 
In accordance with this, the restriction of the application $T$ 
to $\cL_0(N)[X]$, and also that of $\Upsilon$, on the 
sequences side, to finite binary sequences $\cL_0(N)$, are well-defined. 
Furthermore, since as we have seen in the induction process during the proof of Theorem~\ref{TheoremT}, 
the transformations through $\Upsilon$
that occur between the top $\fu$ and the western edge $\fw$ 
actually occur in an ordered manner, 
the first~$N$ components of one only affect the first $N$ components 
of the other, for any $N\ge 0$.
Therefore, Theorem~\ref{TheoremTN}, 
the finite analogue version of Theorem~\ref{TheoremT},
also holds true for~\eqref{eqPGTriangle} triangles of bounded size 
and polynomials instead of formal power series. That is, the restriction of
$T$ to $\cL_0(N)[X]$ is still an involution, just like the 
corresponding restriction of $\Upsilon$ to $\cL_0(N)$ is.

\section{Binary generators and generators with only one champion
}\label{SectionRotationsSymmetriesHelix}

\subsection{Proof of Theorem~\ref{Theorem111}}
A consequence of the fact that $T$ is an involution, as proved in 
Theorems~\ref{TheoremT} and~\ref{TheoremTN}, 
on the side of the coefficients of the formal power series or of the polynomials, 
is the fact that the restriction of $\Upsilon$ to binary sequences 
is also an involution. 
This means that, on one hand, $\Upsilon^{(2)}(\fu) = \fu$ for binary sequences $\fu$
and for any finite initial fragments of these sequences, as well. 
It follows then that $\Upsilon^{(6k)}(\fu) = \fu$ for $k\ge 1$, 
meaning that the helicoid generated by $\fu\in\cL_0$ has all layers identical.

On the other hand, it also follows that $\Upsilon$ is invertible and its inverse
satisfies $\Upsilon^{(-1)} = \Upsilon$. 
Therefore, if $\Upsilon(\fu)=\fw$, it follows that 
$\Upsilon(\fw)=\Upsilon^{(2)}(\fu)=\fu$ 
for all finite or infinite binary sequences $\fu,\fw$. 
Geometrically, this means that the single layer of a helicoid generated by a binary 
sequence is composed of three identical diamond petals. 
Moreover, the diamond is the union of two equilateral triangles, 
positioned symmetrically across the short diagonal~$\fw$. 
This proves the first two parts of Theorem~\ref{Theorem111}.

\begin{figure}
\centering
\includegraphics[width=0.48\textwidth]{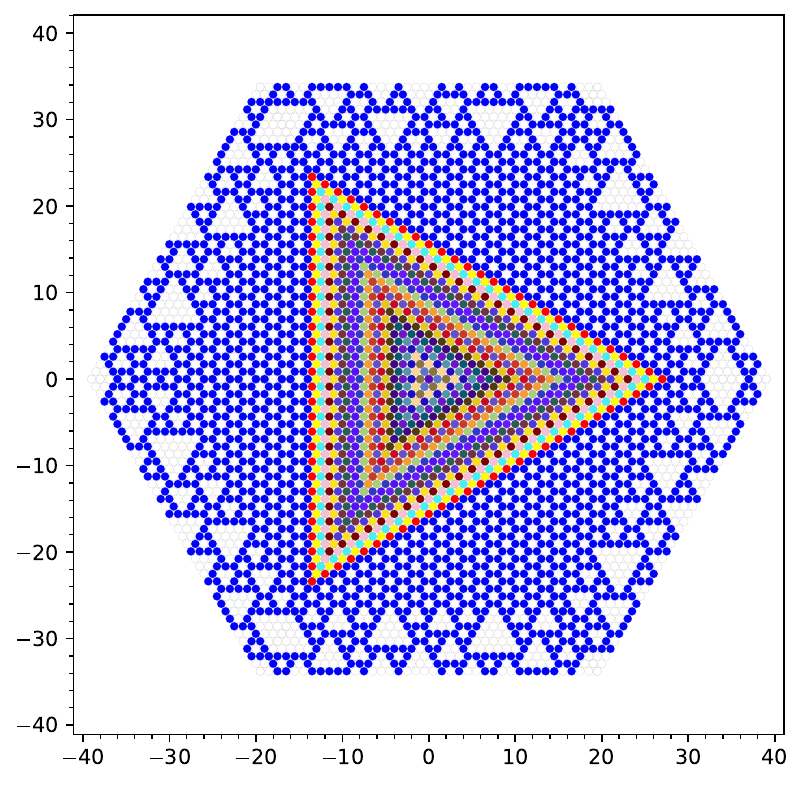}
    \includegraphics[width=0.48\textwidth]{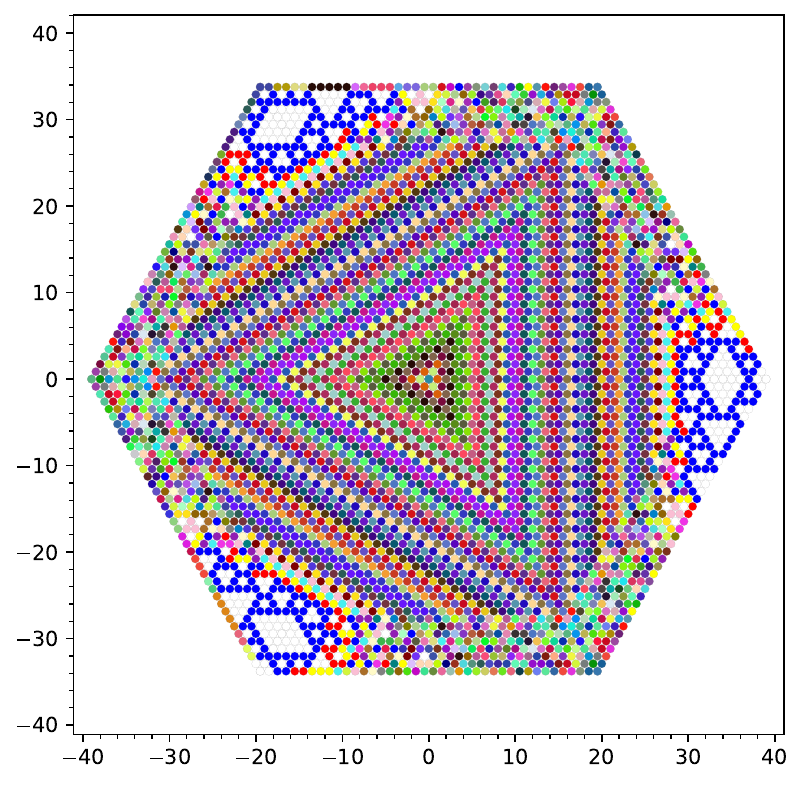}
\vspace{1cm} 
\hfill\hfill
    \includegraphics[width=0.38\textwidth]{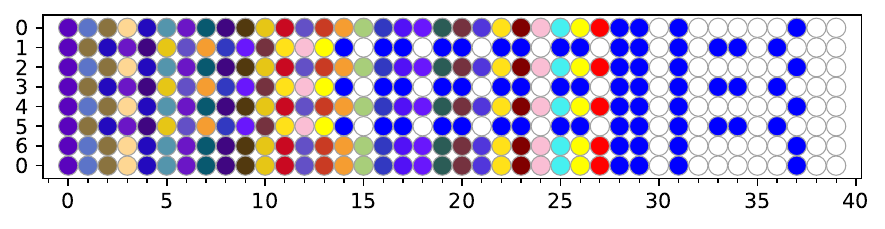}
    \hfill\hfill
    \includegraphics[width=0.38\textwidth]{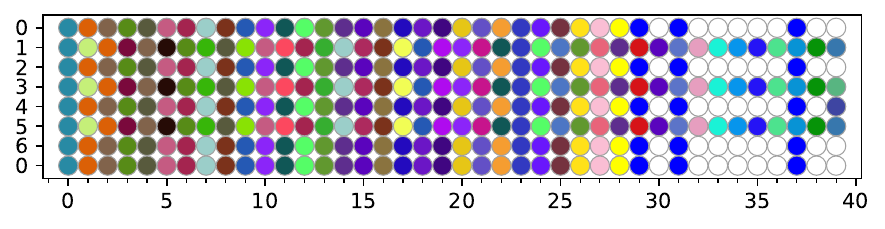}
\hfill\mbox{}
\vspace{-7mm}
\caption{Two helicoids with identical layers on all levels.
They are generated by $\fu$ given by
the first thirty elements of the Fibonacci sequence $F_n$ (left) 
and the Bisection of Fibonacci sequence~$F_{2n}$~\cite[A001906]{oeis} (right) in decreasing order, followed by the sequence of ten bits: $0,1,0,0,0,0,0,1,0,0$. 
Distinct integers are shown in different colors. 
Under the helicoids, the corresponding generating sequences 
$\Upsilon^{(0)}(\fu)$, $\Upsilon^{(1)}(\fu),\dots, \Upsilon^{(6)}(\fu)$
of the intermediate triangles are shown.
In both helicoids, the initial 
$\fu=\Upsilon^{(0)}(\fu)$ is
covered by $\Upsilon^{(6)}(\fu)$, but they can
be seen for comparison on the accompanying maps.
}
\label{FigureFibonacci}
\end{figure}

\begin{figure}
\centering
\includegraphics[width=0.48\textwidth]{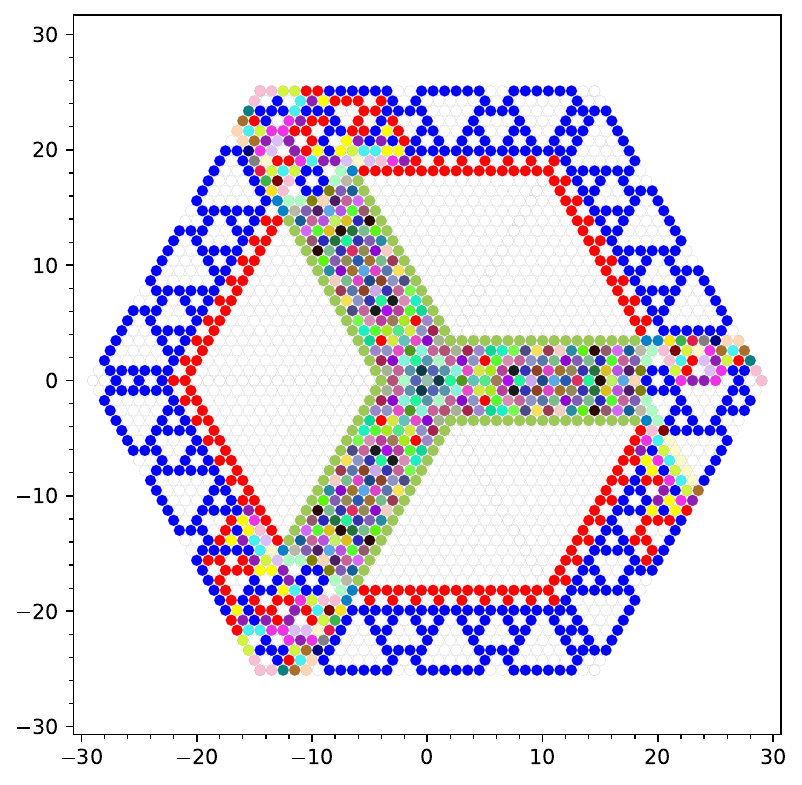}
    \includegraphics[width=0.48\textwidth]{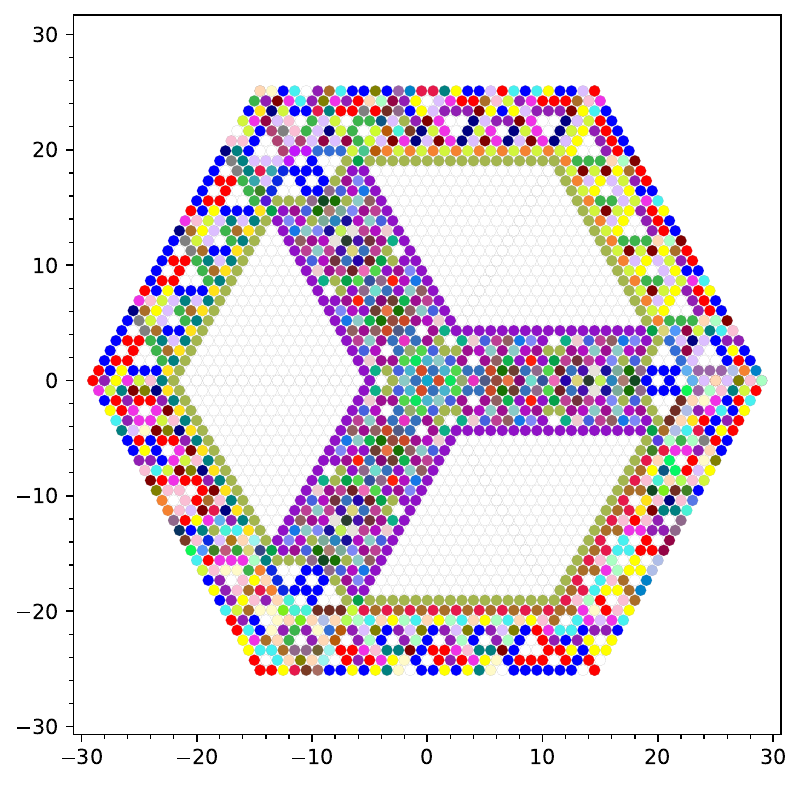}
\vspace{1cm} 
\hfill\hfill
    \includegraphics[width=0.38\textwidth]{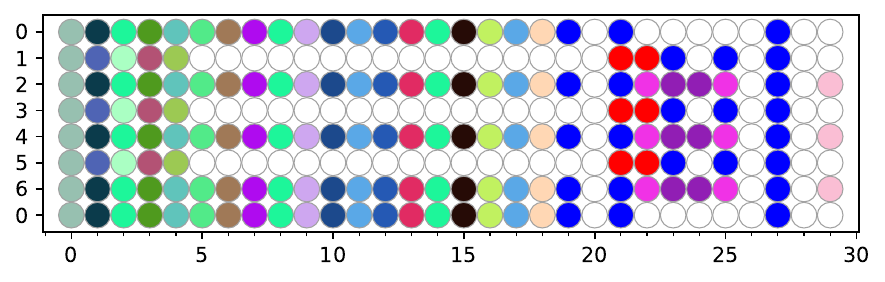}
    \hfill\hfill
    \includegraphics[width=0.38\textwidth]{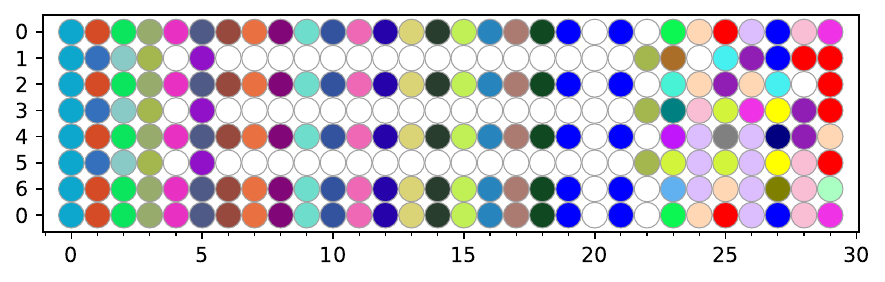}
\hfill\mbox{}
\vspace{-9mm}
\caption{The first level of the helicoid 
generated by $\fu'$ given by
the first twenty $4$th positive powers (left) 
and the fourth level of the helicoid (out of the nine distinct it has)
generated by $\fu''$ given by the 
first twenty $5$th positive powers (right),
both in decreasing order, and then followed each 
by the same sequence of ten bits: $0,1,0,0,0,0,0,1,0,0$. 
Distinct integers are shown in different colors. 
Under the represented levels, the corresponding generating sequences of their intermediate triangles are shown:
$\Upsilon^{(0)}(\fu')$, $\Upsilon^{(1)}(\fu'),\dots, \Upsilon^{(6)}(\fu')$ (left), and
$\Upsilon^{(18)}(\fu'')$, 
$\Upsilon^{(19)}(\fu''),\dots, 
\Upsilon^{(24)}(\fu'')$ (right, numbered also from $0$ to $6$).
In the representation on the left, 
$\fu'=\Upsilon^{(0)}(\fu')$ is 
covered by $\Upsilon^{(6)}(\fu')$
and in the representation on the right
$\Upsilon^{(18)}(\fu'')$ is 
covered by $\Upsilon^{(24)}(\fu'')$, 
but their elements can be seen for comparison 
on the last two rows of the maps underneath.
}
\label{FigurePowers45}
\end{figure}

\begin{figure}
\centering
\includegraphics[width=0.48\textwidth]{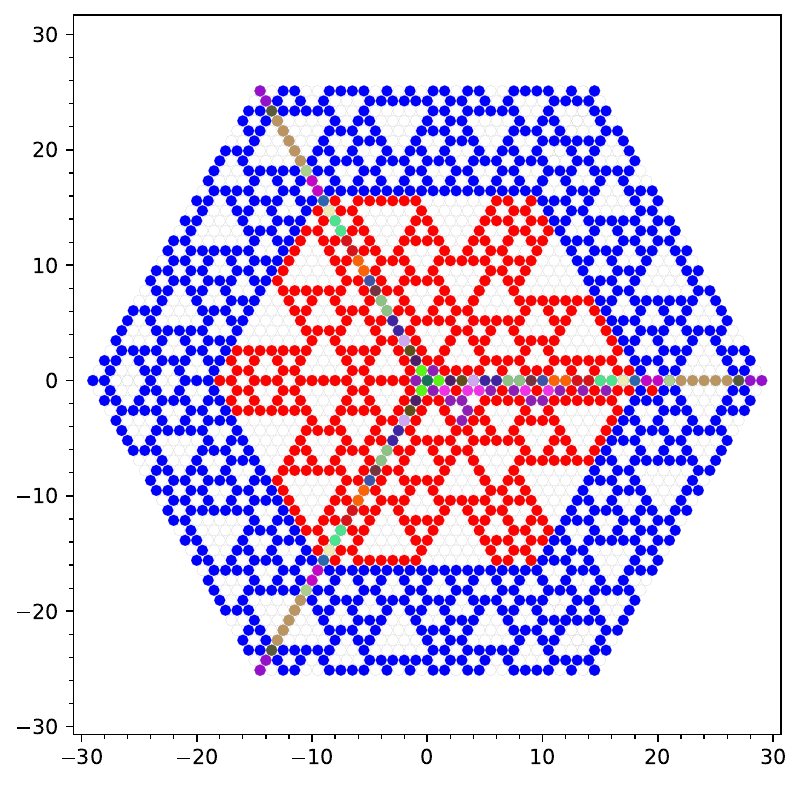}
    \includegraphics[width=0.48\textwidth]{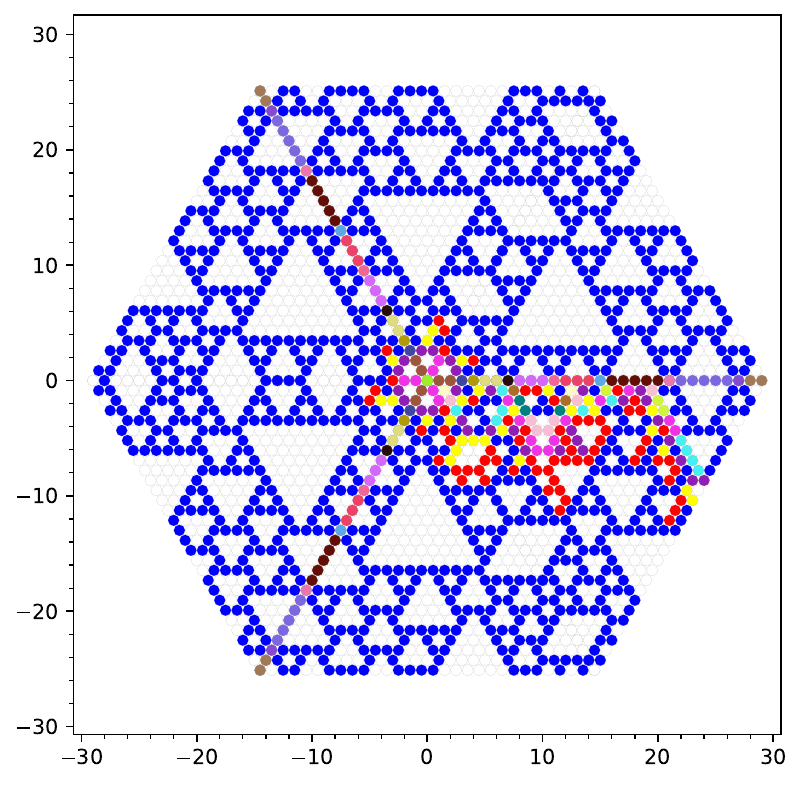}
\vspace{1cm} 
\hfill\hfill
    \includegraphics[width=0.38\textwidth]{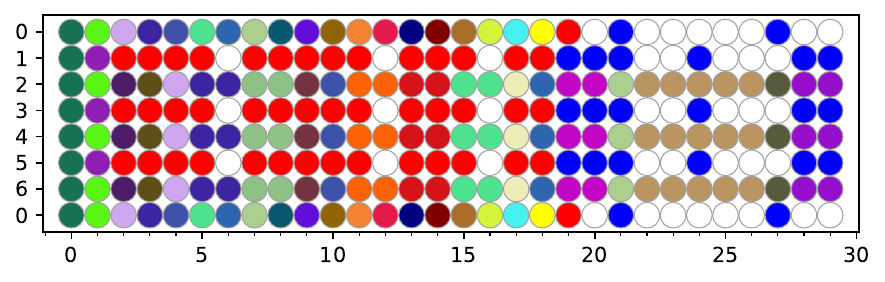}
    \hfill\hfill
    \includegraphics[width=0.38\textwidth]{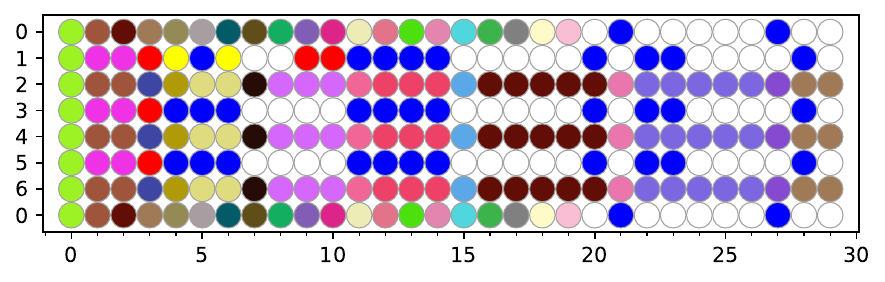}
\hfill\mbox{}
\vspace{-9mm}
\caption{The basic levels of two helicoids that have identical layers 
from the second level on.
They are generated by $\fu$ given by the first twenty primes (left) 
and the first twenty square-primes (right) 
in decreasing order, followed by the sequence of ten bits: $0,1,0,0,0,0,0,1,0,0$. 
Distinct integers are shown in different colors. 
Under the layers, the corresponding generating sequences 
$\Upsilon^{(0)}(\fu)$, $\Upsilon^{(1)}(\fu),\dots, \Upsilon^{(6)}(\fu)$
of the intermediate triangles are shown.
In both images, the initial 
$\fu=\Upsilon^{(0)}(\fu)$ is 
covered by $\Upsilon^{(6)}(\fu)$, but they can
be seen for comparison on the accompanying maps.
}
\label{FigurePrimesSquarePrimes20}
\end{figure}


In order to prove the third part of Theorem~\ref{Theorem111},
let $\rho\ge 0$ be integer and let $\scrC(\rho)$ be the hexagonal
\textit{circle of differences} of radius $\rho$ and center  $a_0$ on a layer generated by $\fu=(a_0,a_1,a_2,\dots)$. 
Explicitly, the numbers on the first edge of the circle are the elements on the eastern edge of the~\eqref{eqPGTriangle} triangle:
\begin{equation*}
    E_0 = \big(a_\rho= d_{\rho}^{(0)}, d_{\rho-1}^{(1)}, d_{\rho-2}^{(2)},
    \dots, d_{0}^{(\rho)}\big),
\end{equation*}
where the differences $d_k^{(j)}$ are defined by~\eqref{eqDifferences}.
Likewise are obtained all the six edges $E_m$, $0\le m\le 5$, of $\scrC(\rho)$, where $E_m$ is the eastern edge of the~\eqref{eqPGTriangle} triangle 
generated by $\Upsilon^{(m)}(\fu)$ instead of $\fu$, that is,
\begin{equation*}
    E_m = \Big(\Upsilon^{(m)}(a_\rho)= d_{\rho}^{(0)}(m), d_{\rho-1}^{(1)}(m),
    d_{\rho-2}^{(2)}(m),
    \dots, d_{0}^{(\rho)(m)}\Big),\ \ \text{ for $m=0,1,\dots,5$,}
\end{equation*}
where
\begin{equation*}\label{eqDifferencesm}
   d_k^{(j+1)}(m) := \big| d_{k+1}^{(j)}(m) - d_{k}^{(j)}(m)\big|
   \quad \text{ and } \quad d_k^{(0)}(m) := \Upsilon^{(m)}(a_k) \quad \text{ for  $j, k\ge 0$.}\\[1mm]
\end{equation*}
Then 
\begin{equation*}\label{eCerc}
   \scrC(\rho) :=E_0\cup E_1\cup E_2\cup E_3\cup E_4\cup E_5\,. 
\end{equation*}

Now, suppose that $\fu$ is the generator of a helicoid that has just one distinct 
layer, that is, we assume that $\Upsilon^{(6)}(\fu)=\fu$.
If $a\ge 0$ is the $\rho$th element of $\fu$, then the assumption says, 
in particular, that the $\rho$th element of $\Upsilon^{(6)}(\fu)$ is also equal to $a$. 
Then, let us analyze the process of generating the layer just on the circle
$\scrC(\rho)$.

\begin{remark}\label{RemarkChampions}
Suppose that $\rho >0$ and $a$ is a champion of the initial sequence $\fu$, that is,
$a>0$ and $a$ is strictly larger than all the elements if  $\fu$ of lower 
indices. 
\begin{enumerate}
\setlength\itemsep{2pt}
    \item 
As the construction of the~\eqref{eqPGTriangle} triangle and its
five subsequent continuations involves only taking absolute values of differences, the numbers on $\scrC(\rho)$ cannot be larger than $a$.
Moreover, the sequence of numbers on $\scrC(\rho)$ cannot increase 
again if it has dropped at any point to a lower value, 
because, by assumption, $a$ is a champion.
    \item 
Since the number that arrives to cover on the upper level the original $a$ after the $6$th
rotation is still~$a$, it then follows that all numbers on circle
$\scrC(\rho)$ are equal.  
    \item 
    Again, since $a$ is a champion, the only possibility for this to happen is 
    when all the numbers on the smaller adjacent
    hexagonal circle $\scrC(\rho -1)$ are $0$'s.
        \item 
    Further, it follows that all the numbers on  $\scrC(\rho -2)$ are also only $0$'s.    
\end{enumerate}
\end{remark}
By iterating, it follows from the above remark that if $a$ is a champion,
then all the numbers on the circle $\scrC(\rho)$ are equal to $a$ and all the 
numbers in the interior of $\scrC(\rho)$ are~$0$'s.
Therefore, there is no other champion in $\fu$ besides $a$, and this concludes
the proof of Theorem~\ref{Theorem111}.

\begin{figure}[htbp]
\centering
\begin{minipage}[b]{0.495\textwidth}
\centering
\includegraphics[width=\textwidth]{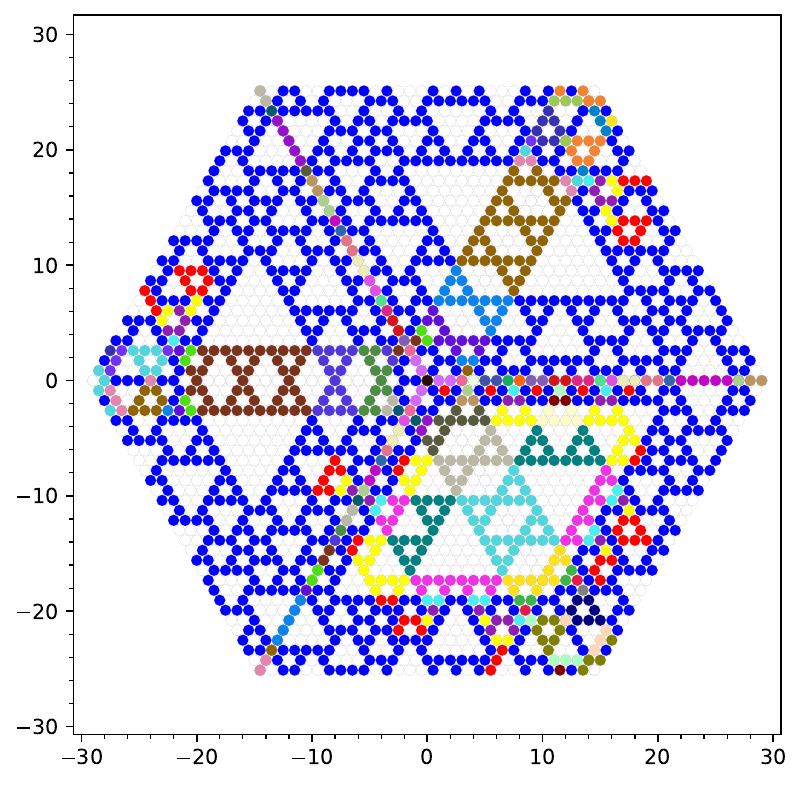}
\end{minipage}
\hfill
\begin{minipage}[b]{0.495\textwidth}
\centering
\includegraphics[width=\textwidth]{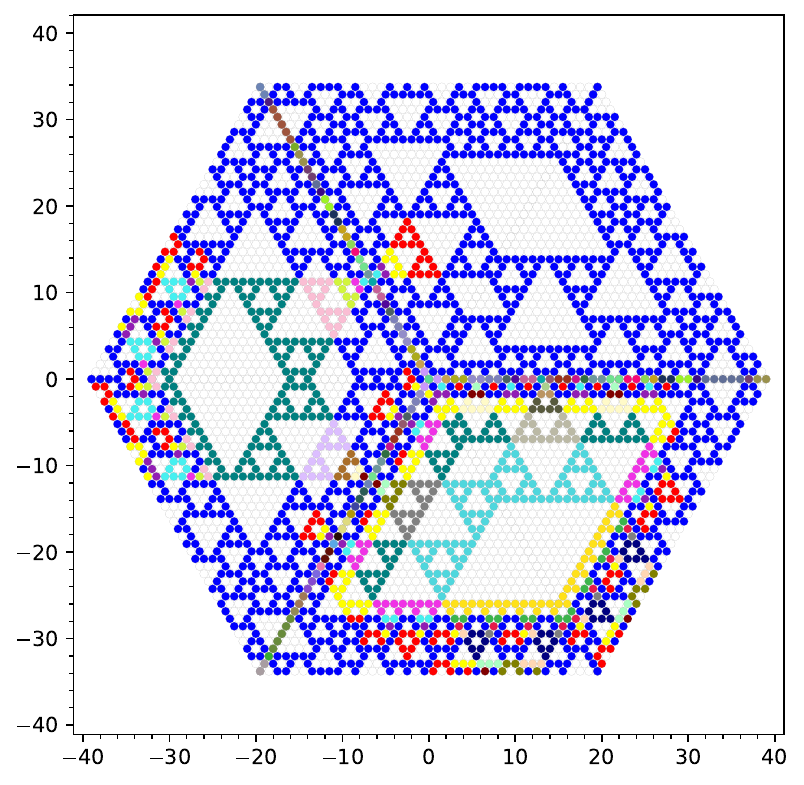}
\end{minipage}
\vspace{2mm}
\begin{minipage}[b]{0.495\textwidth}
\centering\ \ \
\includegraphics[width=.72\textwidth]{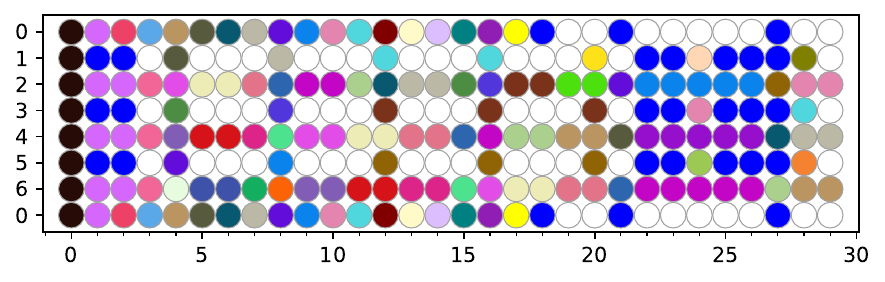}
\end{minipage}
\hfill
\begin{minipage}[b]{0.495\textwidth}
\centering\quad
\includegraphics[width=.92\textwidth]{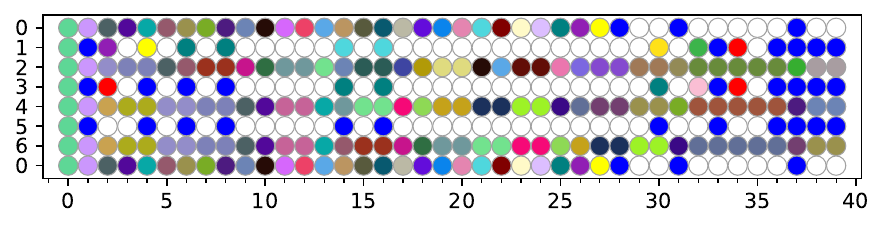}
\end{minipage}
\caption{The hexagons on the base level of two helicoids generated by
the sequence $\fu$ of non-negative integers whose base-$3$ representation contains 
no $2$~\cite[A005836]{oeis}. The image on the left uses $20$ and the image 
on the right $30$ elements of the sequence in decreasing order, both 
followed by the same ten bits: $0,1,0,0,0,0,0,1,0,0$. 
Distinct integers are shown in different colors. 
The helicoids have four and two distinct levels, respectively.
Under the hexagons, the corresponding generating sequences 
$\Upsilon^{(0)}(\fu)$, $\Upsilon^{(1)}(\fu),\dots, \Upsilon^{(6)}(\fu)$
of the intermediate triangles are shown.
In both images, the initial generators
$\fu=\Upsilon^{(0)}(\fu)$ are 
covered by $\Upsilon^{(6)}(\fu)$, but they can
be seen for comparison in the accompanying maps.}
\label{FigureAP3}
\end{figure}
\subsection{Trial of sequences with a single champion}
The necessary condition for sequences to have at most one champion in order for the
helicoids they generate to have just one distinct layer  
proves to be insufficient. In fact, there may exist one-champion sequences 
that generate helicoids with a record number of distinct layers.
In Figures~\ref{FigureFibonacci}--\ref{FigureAP3}, 
the helicoids generated by different finite integer sequences are shown, 
each having a single champion, their first element. 
To use comparable units in reasonably sized images that can be displayed in print, 
we have chosen decreasing sequences $\fu$ of $20$ or $30$ integers, 
all followed by the same sequence of $10$ random bits: $0,1,0,0,0,0,0,1,0,0$. 
Distinct numbers are represented by different colors. 
Under each helicoid, the generating sequences of the partial equilateral triangles, 
namely $\Upsilon^{(0)}(\fu)$, $\Upsilon^{(1)}(\fu),\dots, \Upsilon^{(6)}(\fu)$,
$\Upsilon^{(0)}(\fu)$, can be seen stacked on top of each other, 
making it easier to compare and determine if the helicoid has multiple distinct sheets on distinct levels. 

One finds that the results are mixed. 
There are sequences that generate helicoids with one distinct level, 
as the one in~Figure~\ref{FigureFibonacci}, or with exactly two distinct levels, as those in~Figure~\ref{FigurePowers45} (left),
Figure~\ref{FigurePrimesSquarePrimes20} and Figure~\ref{FigureAP3} (right).
There are also helicoids with more distinct levels, like the one
that is generated by $5$th powers in~Figure~\ref{FigurePowers45} (right), 
where is shown level four out of the nine distinct ones it has,
or the one in  Figure~\ref{FigureAP3} (left), 
which has four distinct levels.
The lead sequence, also used in reversed order in Figure~\ref{FigureAP3}, is  
$0, 1, 3, 4, 9, 10, 12, 13, 27, 28, 30, 31,\dots$,
the sequence that starts with $0$ and is generated by the greedy algorithm so that
it contains no arithmetic progressions of length $3$.   
Its elements are also characterized as being sums of distinct powers of $3$ 
or as having only $0$ and $1$ in their base-$3$ representation~\cite[A005836]{oeis}.
The longer instances with the first  $N=314, 315,\dots, 320$ elements
in reversed order
followed by the above sequence of ten bits
generate helicoids with $84$ distinct layers, the cycle having
just one element.

A complete characterization of the integer sequences based on the number of distinct levels that their associated helicoids have requires further investigation. 
For instance, it would be interesting to know if there are helicoids with an arbitrarily large number of distinct levels.
A promising candidate to try seems to be a series of length-balanced 
decreasing sequences of powers. 
For instance, to add to the points already mentioned, we note that
the sequence of just ten $9$th powers in decreasing order followed by $0,1,0,0,0,0,0,1,0,0$
produces a helicoid with $262$ distinct layers, of which $P=198$ are in a precycle and $C=64$
are in an endless repeated cycle.
And the analogues with $10$th and $11$th powers have $P+C = 140+128=268$ 
and  $P+C = 512+32=544$ distinct layers, respectively.
\begin{question*}
 Is there a sequence of finite sequences of positive integers that generates 
 helicoids with an unlimited number of distinct levels?
\end{question*}


\section{Proof of Theorem~\ref{TheoremAlmostAllSeq}}\label{SectionProofOfTheoremAlmostAllSeq}
Let us start by noting that it is sufficient to prove the belonging relationship 
in~\eqref{TheoremAlmostAllSeq} for proportions $R_\fw(k)$, as it will also imply the one for 
$R_\fe(k)$. 
This follows from the rotation-reflection symmetry, as we have seen in the discussion
Section~\ref{SectionRotationsSymmetriesHelix}, and also by following the same reasoning below 
with the sequences in $\fu=\fu(N)=(a_0,\dots,a_{N-1})\in\cL_0(N)$ indexed from right to left instead of left to right, 
which obviously would lead to the same conclusion, this time for $R_\fe(k)$.

Let $\varepsilon\in (0,1/2)$ be fixed. We first prove that there exists an integer $N_\varepsilon$
such that if $N\ge N_\varepsilon$ then
\begin{equation}\label{eqProofT11}
   \begin{split}
    \sdfrac{1}{2^N}\#\Big\{
    \fu\in\cL_0(N) :
    \sdfrac{1}{N}\#\big\{0\le j\le N-1 : a_j=1\big\}
    \in\big[1/2-\varepsilon,\,1/2+\varepsilon\big]
    \Big\}
    \ge 1-\varepsilon.
   \end{split}	
\end{equation}
For each $\fu\in\cL_0(N)$, denote the set of indices with components equal to $1$ by
\begin{equation*}
   S(\fu) := \big\{0\le j\le N-1 : a_j=1\big\} \subset\{0,1,\dots,N-1\}.
\end{equation*}
Then the left-hand side of inequality~\eqref{eqProofT11} can be rewritten as
\begin{equation*}\label{eqProofT12}
   \begin{split}
    &\sdfrac{1}{2^N}\#\Big\{
    \fu\in\cL_0(N) :
    \sdfrac{1}{N}\#S(\fu)
    \in\big[1/2-\varepsilon,\,1/2+\varepsilon\big]
    \Big\}\\
    =&\sdfrac{1}{2^N}\#\Big\{
    \fu\in\cL_0(N) :
    \sdfrac{N}{2}-N\varepsilon\le \#S(\fu) \le \sdfrac{N}{2}+N\varepsilon
    \Big\}\\
     =&\sdfrac{1}{2^N} \sum_{\frac{N}{2}-N\varepsilon\le l\le \frac{N}{2}+N\varepsilon}
     \#\Big\{
    \fu\in\cL_0(N) : \#S(\fu)=l\Big\}\,.
   \end{split}	
\end{equation*}
Noticing in passing that the unrestricted sum of all the cardinalities in the last sum 
is $2^N$, the inequality~\eqref{eqProofT11} that we want to prove is
\begin{equation*}
   \begin{split}
   \sdfrac{1}{2^N} \sum_{\frac{N}{2}-\varepsilon N\le l\le \frac{N}{2}+\varepsilon N}
    \binom{N}{l}\ge 1-\varepsilon\,
   \end{split}	
\end{equation*}
for sufficiently large $N$.
For this, it is sufficient to prove that in the sum above, the missed tails at the beginning and at the end, 
which are equal, due to the symmetry of the binomial coefficients, are small, that is,
\begin{equation*}
   2 \sum_{0 \le l\le M}
    \binom{N}{l}\le \varepsilon 2^N\,,
\end{equation*}
where we denoted $M:=\big\lfloor \frac{N}{2}-\varepsilon N\big\rfloor$.
Since the binomial coefficients are increasing in the range of summation above, we replace the sum with 
 a trivial upper bound and reformulate our object as the following 
 convenient statement that is sufficient to be proven:
\begin{equation}\label{eqProofT14}
   (M+1) \binom{N}{M}\le \varepsilon 2^{N-1}\,,
\end{equation}
for sufficiently large $N$.
To estimate the binomial coefficients we use Stirling's approximation formula for factorials given by Robbins~\cite{Rob1955} 
in the form of two tight upper and lower bounds:
\begin{equation*}
\sqrt {2\pi n}\left({\frac {n}{e}}\right)^{n}e^{\frac {1}{12n+1}}
<n!<
{\sqrt {2\pi n}\left({\frac {n}{e}}\right)^{n}e^{\frac {1}{12n}}.}
\end{equation*}
Then the left-hand side of~\eqref{eqProofT14} is 
\begin{equation}\label{eqProofT15}
   (M+1) \binom{N}{M} = (M+1) \frac{N!}{M!(N-M)!}
   < a 
   \frac{N^{N+1/2}}{M^{M-1/2}(N-M)^{N-M+1/2}},
\end{equation}
for a positive constant $a<1/200$ if $N>100$.
Here, we change the variable $M$ into $R$, where $R:=N/2-M$. 
Note that since $M=N/2-\varepsilon N + \theta$, where $|\theta |\le 1$, it follows that
\begin{equation}\label{eqR}
    R=\varepsilon N +O(1).
\end{equation}
Then the new form of the inequality~\eqref{eqProofT14} that we want to prove for 
$N$ sufficiently large is
\begin{equation*}
   a\cdot
   \frac{N^{N+1/2}}{\Big(\frac N2-R\Big)^{N/2-R-1/2}  \Big(\frac N2+R\Big)^{N/2+R+1/2}}
   \le \varepsilon 2^{N-1},
\end{equation*}
which can still be rearranged further into the more convenient form
\begin{equation}\label{eqProofT16}
   2a\cdot
   \frac{N^{1/2}}{\Big(1-\frac{2R}{N}\Big)^{N/2-R-1/2}  \Big(1+\frac{2R}{N}\Big)^{N/2+R+1/2}}
   \le \varepsilon. 
\end{equation}
Let us note that here $r:=2R/N$ is small because~\eqref{eqR} implies
\begin{equation}\label{eqr}
    r=\frac{2R}{N} 
    = \frac{2\varepsilon N+O(1)}{N} 
    = 2\varepsilon +O\Big(\sdfrac{1}{N}\Big).
\end{equation}
Introducing the new variable $r$ in~\eqref{eqProofT16}, we find that it is equivalent with
\begin{equation}\label{eqProofT17}
   2a\cdot
   \frac{N^{1/2}}{\big(1-r^2\big)^{N/2}  (1-r)^{-R-1/2} (1+r)^{R+1/2}}
   \le \varepsilon. 
\end{equation}
It remains to show that here the denominator has an order of magnitude higher than that of the numerator. 
To do this, we evaluate the logarithm of the denominator, which is
\begin{equation*}
  \begin{split}
   =&  \log\Big(\big(1-r^2\big)^{N/2} (1-r)^{-R-1/2} (1+r)^{R+1/2}\Big)\\
   =&
   \sdfrac{N}{2}\log\big(1-r^2\big) +\Big(R+\sdfrac 12\Big) \big(\log(1+r)-\log(1-r)\big).
  \end{split}
\end{equation*}
Then, by taking into account the size of $R$ and $r$ 
from~\eqref{eqR} and~\eqref{eqr}, and particularly the fact that $r$ is small,
we replace the logarithms with their power series approximations and
see that the logarithm of the denominator in~\eqref{eqProofT17} is further equal to
\begin{equation}\label{eqProofT19}
  \begin{split}
     &\sdfrac{N}{2}\big(-r^2+O(r^4)\big) 
     + \Big(R+\sdfrac 12\Big)\Big(2r+\sdfrac{2r^3}{3}+O(r^5)\Big) \\
     = & N\sum_{k=1}^\infty 2^{2k}
     \left(\frac{1}{2k-1}-\frac{1}{2k}\right) \varepsilon^{2k}
     +O(\varepsilon).
  \end{split}
\end{equation} 
It then follows that there exists an absolute constant $c>0$ such that~\eqref{eqProofT17} is satisfied for sufficiently large $N$ if
\begin{equation*}
   \begin{split}
   \frac{\log N}{c \varepsilon^2 N} < \frac{N^{1/2}}{c \varepsilon^2 N} 
   = \frac{1}{c\varepsilon^2N^{1/2}} <\varepsilon.
   \end{split}   
\end{equation*}
And, in order for this last requirement to be fulfilled, it is enough to take 
$N>N_\varepsilon$, with $N_\varepsilon:=\max\big\{100,\,\big\lfloor\frac{1}{c\varepsilon^6}\big\rfloor\big\}$.

To conclude the proof of Theorem~\ref{TheoremAlmostAllSeq} for just the left edge of 
the~\eqref{eqPGTriangle} triangle, we apply the operator $\Upsilon$ on the generating
sequences $\fu$.
Since we know by Theorems~\ref{TheoremT} and~\ref{TheoremTN}
that $\Upsilon$ is an involution, therefore also a bijection, it follows from the above argument
that, aside from an exceptional set of the same size as that of $\fu$'s,
on the western edge, the sequences~$\fw_0$ contain approximately the same number of $0$'s and $1$'s as well.

\smallskip

Then, in the same way, the same conclusion can be drawn for the next 
segments $\fw_1,\fw_2,\dots$, which are parallel to $\fw_0$, 
by reasoning with the partial subsequences of $\fu$ that ignore 
a few starting elements. 
In the following,  we need to quantify precisely for how many segments we can be sure that 
the almost-equal-proportions-result actually holds for.

Let $K$ be the number of the rays $\fw_0,\fw_1,\dots,\fw_{K-1}$ in question.
The size of $K$ that we can afford will be determined later.
These rays are generated by the partial subsequences of $\fu$ that are given by
$\fu_k=\fu_k(N):=(a_k,a_{k+1},\dots,a_{N-1})$ for $0\le k\le K-1$. 
Note that the number of the sequences $\fu_k(N)$ is $2^{N-k}$ for any fixed $k$ and $N$.

Let $\varepsilon\in (0,1/2)$ be fixed and let $k$ and $K$ be such that
$0\le k\le K-1<N$ for some large~$N$.
In order to fulfill the additional requirements on the multiple rays, 
we choose a larger value for the term on the right side of the inequality~\eqref{eqProofT11}, which means, a narrower target.
Thus, since the size of~$\fu_k$ is $N-k$,
its counterpart form becomes:
\begin{equation}\label{eqProofT11k}
   \begin{split}
    \sdfrac{1}{2^{N-k}}\#\Big\{
    \fu\in\cL_0(N-k) :
    \sdfrac{\#\big\{k\le j\le N-1 : a_j=1\big\}}{N-k}
    \in\big[1/2-\varepsilon,\,1/2+\varepsilon\big]
    \Big\}
    \ge 1-\sdfrac{\varepsilon}{K}.
   \end{split}	
\end{equation}
Let us note that this would suffice to completely prove  Theorem~\ref{TheoremAlmostAllSeq}. Indeed,  let $\cE_N(k)$ be the part of the exceptional set $\cE_N$ in the statement of
Theorem~\ref{TheoremAlmostAllSeq} that corresponds to the sequences $\fu$ 
that do not meet the requirement~\eqref{eqTheoremAlmostAllSeq} of 
approximately-equal-proportion of~$1$'s and $0$'s  for ray $\fw_k$.
Then, the exceptional sets are smaller and smaller in size as $k$ increases,
and despite the fact that the sets $\cE_N(k)$ are not disjoint,
we have
\begin{equation}\label{eqES}
    \cE(N) = \bigcup_{k=0}^{K-1}\cE_N(k)
    \quad\text{ and }\quad
    \#\cE(N)\le \sum_{k=0}^{K-1}\#\cE_N(k).
\end{equation}

We know, via the one-to-one correspondence $T$ from Theorem~\ref{TheoremTN}, that the  generating sequence $\fu$ and the left edge $\fw$ of any typical~\eqref{eqPGTriangle} 
triangle of zeros and ones are permutations of one another, 
so that, the cardinality of the analogous exceptional sets of
sequences that do not fulfill the approximately equal number of $1$'s and $0$'s condition
is the same.

Therefore, once proven~\eqref{eqProofT11k}, using the inequality in~\eqref{eqES}, 
for the exceptional set $\cE(N)$ in the statement of Theorem~\ref{TheoremAlmostAllSeq}, we find that
\begin{equation*}
    \begin{split}
    \frac{1}{2^N}\#\cE(N) \le 
     \frac{1}{2^N} \sum_{k=0}^{K-1}\#\cE_N(k)
     \le \sum_{k=0}^{K-1} \frac{1}{2^{N-k}}\#\cE_N(k)
     \le K\cdot \frac{\varepsilon}{K}=\varepsilon,
    \end{split}	
\end{equation*}
that is, $\#\cE(N)\le \varepsilon 2^N$, as needed.
Then we only have to prove~\eqref{eqProofT11k}, indicating the range of $N$ 
in which it holds, and how large we are allowed to take $K$.

Following the same steps as before in arguing relation~\eqref{eqProofT11}, 
by rewriting the cardinality of the sets inside~\eqref{eqProofT11k}, 
and then continuing by simplifying the resulting expression, 
the inequality on the tails becomes
\begin{equation*}
   \frac{1}{2^{N-k-1}} \sum_{0 \le l\le M}
    \binom{N-k}{l}\le \frac{\varepsilon}{K}\,,
\end{equation*}
where $M=\big\lfloor \frac{N-K}{2}-\varepsilon (N-K)\big\rfloor$.
It then follows that it suffices to see that for large $N$ we have
\begin{equation}\label{eqProofT144}
   (M+1) \binom{N}{M}\le \frac{\varepsilon}{K} 2^{N-1-K}\,,
\end{equation}
the strongest condition that, once met, 
all the corresponding $k$ inequalities will also 
be satisfied for $0\le k\le K-1$.

Further, the estimates of the binomial coefficients is done as
before, and we arrive to ponder the inequality
\begin{equation}\label{eqProofT177}
   \frac{2^{K}N^{1/2}}{\big(1-r^2\big)^{N/2}  (1-r)^{-R-1/2} (1+r)^{R+1/2}}
       \le \frac{\varepsilon}{AK},
\end{equation}
for some constant $A>0$ that is independent of $\epsilon, N$ and $K$.
Here $r$ is small and $R$ captures the new $M$ from~\eqref{eqProofT144}.
Precisely, as in~\eqref{eqR} and~\eqref{eqr}, we have:
\begin{equation}\label{eqRr}
    r = 2\varepsilon+O\Big(\sdfrac{1}{N}\Big)
    \quad\text{ and }\quad
    R = \varepsilon N+ O(1).
\end{equation}

The logarithm of the numerator on the left-hand side of~\eqref{eqProofT177} is
\begin{equation}\label{eqlogSUS} 
    K\log 2+\frac 12\log N
\end{equation}
and, using~\eqref{eqRr}, the logarithm of the denominator is the same as in~\eqref{eqProofT19}:
\begin{equation}\label{eqlogJOS} 
     N\sum_{k=1}^\infty 
     \frac{2^{2k-1}}{k(2k-1)} \varepsilon^{2k}
     +O(1) > c_0 \varepsilon^2 N.
\end{equation}
for some absolute constant $c_0>0$.

As a consequence, using~\eqref{eqlogSUS} and~\eqref{eqlogJOS}, we find that   inequality~\eqref{eqProofT177} holds true for sufficiently large $N$ as long 
as the following condition is also verified.
\begin{equation*}
    c_0\varepsilon^2 N
    -K\log 2-\frac 12\log N > \log\frac{K}{\varepsilon}
\end{equation*}
But, for this to happen, it is necessary to exist absolute constants
$c_1,c_2>0$ such that
\begin{equation*}
    c_1\varepsilon^2 N > K
    \quad \text{ and }\quad 
    c_2\varepsilon^2 N> \log {K} - \log{\varepsilon}.
\end{equation*}
Both these conditions are fulfilled if we take $K=\delta_\varepsilon N$, 
where $\delta=\delta_\varepsilon>0$ is
a suitable small constant that depends only on $\varepsilon$ and the absolute constants above, and $N$ is larger than a threshold~$N_{\varepsilon,\delta}$ that relies on the same dependencies and additionally on the choice of $\delta$.
This concludes the proof of Theorem~\ref{TheoremAlmostAllSeq}.

\medskip
\subsection*{Acknowledgement}
The authors acknowledge the contributions of 
Sundarraman Madhusudanan, a co-author of \cite{BMa2022},
in the proof of Lemma~\ref{LemmaSP2}.
%


\relax
\bibliographystyle{myabbrv}



\end{document}